\documentclass[oneside, 11pt]{amsart}
\usepackage{amsmath}
\usepackage{amsfonts}
\usepackage{amssymb}
\usepackage{mathrsfs}
\usepackage{amsthm}
\usepackage{graphicx}
\usepackage{xcolor}
\usepackage{array}
\usepackage{comment}
\usepackage[margin=1in]{geometry}
\usepackage{hyperref}
\usepackage{float}
\usepackage{caption}
\usepackage{comment}
\usepackage[makeroom]{cancel}

\newtheorem{theorem}{Theorem}[section]
\newtheorem*{theorem*}{Theorem}
\newtheorem*{maintheorem*}{Main theorem}
\newtheorem{remark}[theorem]{Remark}
\newtheorem{example}[theorem]{Example}
\newtheorem{lemma}[theorem]{Lemma}
\newtheorem{proposition}[theorem]{Proposition}
\newtheorem{corollary}[theorem]{Corollary}
\newtheorem*{corollary*}{Corollary}

\newtheorem{claim}{Claim}[section]

\numberwithin{equation}{section}

\DeclareMathOperator{\Tr}{Tr}

\newcommand{\ad}{\mathrm{ad}}
\newcommand{\Ad}{\mathrm{Ad}}
\newcommand{\Adstar}{\Ad^{\star}}
\newcommand{\adstar}{\ad^{\star}}

\newcommand{\id}{\mathrm{Id}}

\newcommand{\Id}{\mathrm{Id}}

\newcommand{\so}{\mathfrak{so}}

\newcommand{\g}{\mathfrak{g}}
\newcommand{\h}{\mathfrak{h}}

\newcommand{\hperp}{\mathfrak{h}^{\perp}}
\newcommand{\verproj}{\mathscr{V}}
\newcommand{\horproj}{\mathscr{H}}
\newcommand{\consty}{D}
\newcommand{\bim}{\kappa}
\newcommand{\subby}{\Omega}

\newcommand{\simpkill}{\kappa}
\newcommand{\Syl}{S}
\newcommand{\eulerop}{F}
\newcommand{\adop}{A}
\newcommand{\adlambop}{B}

\newcounter{eq_counter_save}

\newcommand{\mattyone}{X_1}
\newcommand{\mattytwo}{X_2}
\newcommand{\unit}{E}
\newcommand{\son}{e}
\newcommand{\heis}{e}

\author[1]{Alice Le Brigant} 
\author[2]{Leandro Lichtenfelz}
\author[3]{Stephen C. Preston}
\address{A. Le Brigant: Paris 1 Panthéon-Sorbonne University, France}
\email{alice.le-brigant@univ-paris1.fr}
\address{L. Lichtenfelz: Wake Forest University, USA}
\email{lichtel@wfu.edu}
\address{S.C. Preston: Brooklyn College and CUNY Graduate Center, USA}
\email{stephen.preston@brooklyn.cuny.edu}

\date{}

\begin{document}
 
\title{Conjugate points on Lie groups with left-invariant metrics}

\begin{abstract}
We prove sufficient conditions for the existence of conjugate points along geodesics of a left-invariant metric on a Lie group, using a reformulation of the index form in terms of the adjoint action. In the compact semisimple case, with an arbitrary left-invariant metric, we show that all geodesics must have a conjugate point, and we give upper and lower bounds on conjugate times. In particular this applies to the left-invariant metrics on $SU(N)$ and $SO(N)$ which are of importance in fluid dynamics and rigid body motion, and yields estimates for the diameter and injectivity radius. We also establish criteria in the noncompact case: we show that every closed nonhomogeneous geodesic has a conjugate point, and determine explicit conditions for them in the three-dimensional unimodular case. For homogeneous geodesics, we relate conjugate points to Lagrangian stability, and Eulerian stability of the corresponding steady velocity. Finally, we obtain as by-products criteria for conjugate points in general homogeneous spaces, by lifting the problem to the total Lie group of the quotient and using a result of O'Neill. Through several examples, we show that our theorems apply when well-known criteria relying on positive Ricci curvature or other curvature bounds fail, and in some cases even when Ricci curvature is negative along a certain geodesic.
\end{abstract}

\maketitle

\textbf{MSC2020:} 53C30, 58E10, 70G65
\setcounter{tocdepth}{1}

\tableofcontents

\section{Introduction}

One of the most basic questions in Riemannian geometry concerns the length-minimization properties of geodesics. In particular, the presence of conjugate points indicates a geodesic’s failure to minimize length past some time, and their location can yield a bound for the diameter or the injectivity radius. When the manifold represents the configuration space of some physical object, such as a rigid body or a fluid, 
the study of conjugate points sheds light on the way the trajectories of these objects behave under perturbation of their initial conditions. Lie groups, and the homogeneous spaces obtained as their quotients, provide convenient models for such configuration spaces. Classical examples include the kinetic energy metric on $SO(N)$ describing the motion of a free rigid body in $N$-dimensional space \cite{Frahm}, and Zeitlin's finite-dimensional model of spherical hydrodynamics on $SU(N)$ ~\cite{Zeitlin}.

In a Lie group equipped with a left-invariant (or right-invariant) metric, the geodesic equation can be translated to the Lie algebra, to give the so-called Euler--Arnold equation. This induces a distinction between homogeneous geodesics, corresponding to steady solutions of the Euler--Arnold equation, and nonhomogeneous geodesics, whose left-translated velocity at the origin is nonconstant in time. A conjugate point occurs along any such geodesic when there exists a Jacobi field that vanishes at both end points. In certain special homogeneous spaces, such as symmetric, normal or naturally reductive spaces, the Jacobi equation can be reduced to a constant coefficient equation and the problem becomes tractable \cite{Chavel, Chavel2, Chavel3, herring, Ziller, Ziller2}. However, in general Lie groups and along nonhomogeneous geodesics, the Jacobi equation has nonconstant coefficients and is generally impossible to solve explicitly. Another strategy to find conjugate points is through the study of the index form, a quadratic form defined in terms of the formula for the second variation of the energy. The goal is to exhibit a perturbation along the geodesic for which this quadratic form is negative, thereby proving the existence of a vanishing Jacobi field without actually having to find it. This is the strategy behind the Misio{\l}ek criterion, a sufficient condition for the existence of a perturbation that makes the index form negative after some large enough time, initially introduced in the infinite-dimensional setting of hydrodynamics~\cite{M1}. Although it can be readily generalized to any Lie group, this criterion is not exhaustive and has been shown to fail to detect conjugate points in some settings (see~\cite{lbpreston}).
 
As can easily be deduced from the index form criterion, there needs to be at least \emph{some} positive sectional curvature for conjugate points to appear. A result of Gromoll and Meyer \cite{GM} states that if the sectional curvatures are positive \emph{on average}, meaning that the Ricci curvature is nonnegative along the geodesic and positive on at least one point, then conjugate points necessarily occur. On a Lie group, this makes the existence of conjugate points trivial for a bi-invariant metric, due to its nonnegative sectional curvature. With only one-sided invariance on the other hand, as illustrated in the well-known paper of Milnor~\cite{Milnor}, the curvature typically takes on both signs. 

In this paper, we establish several new sufficient criteria for the existence of conjugate points in a Lie group with a left-invariant (or right-invariant) metric. Using a result from O'Neill on Riemannian submersions, we obtain as by-products criteria for conjugate points in general homogeneous spaces, without the need for stronger assumptions (such as normality or natural reductiveness). Taking advantage of the natural splitting of the geodesic equation---into the flow equation and the Euler--Arnold equation---as well as the splitting of the Jacobi equation, we prove criteria that do not rely on standard curvature assumptions, and detect conjugate points when other known criteria fail. In particular, we show that conjugate points can appear even when the negative sectional curvature outweighs, in the sense of Ricci, the positive sectional curvature along the geodesic.

\subsection{Main result}

Our main contribution concerns the case where the group adjoint is uniformly bounded. Using a reformulation of the index form (Lemma \ref{lemma_index}), we show that under the uniform boundedness assumption, a conjugate point occurs along a geodesic if and only if its initial velocity $u_0$ is not orthogonal to the commutator subalgebra $[\mathfrak{g},\mathfrak{g}]$ (Theorem~\ref{generaltheorem}), and we give bounds for the conjugate time when the group admits a bi-invariant metric (Corollary~\ref{easyconjugatethm}). In particular for compact semisimple Lie groups, we get the following result.
\begin{maintheorem*}
Every compact semisimple Lie group equipped with a left-invariant (or right-invariant) metric $\langle\cdot,\cdot\rangle$ admits conjugate
points along any geodesic. Furthermore, the first conjugate time $\tau$ can be estimated by
$$\frac{2\pi \lambda_m}{\lVert \ad_{\Lambda u_0}\rVert_{op}} \leq  \tau \leq \frac{2\pi \lambda_M}{\lVert \ad_{\Lambda u_0}\rVert_{op}},$$
where $u_0$ is the initial velocity, $\Lambda$ is the inertia operator of the one-sided %left-
invariant metric with respect to a bi-invariant metric $\bim(\cdot, \cdot)$ such that $\langle \cdot, \cdot\rangle = \bim(\cdot, \Lambda\cdot)$, $\lambda_m$ and $\lambda_M$ are the smallest and largest eigenvalues of $\Lambda$ resp., and the operator norms are computed using $\bim(\cdot, \cdot)$.
\end{maintheorem*}
Using O’Neill’s theorem on Riemannian submersions, we obtain an analogous existence theorem for conjugate points along geodesics on homogeneous spaces. 
\begin{corollary*}
If $G$ is a compact semisimple Lie group with any left-invariant metric and $H$ is any closed subgroup, then every geodesic in the homogeneous space $G/H$ develops a conjugate point.
\end{corollary*}
The above main theorem applies in particular to $SO(N)$, the configuration space of rigid bodies, and $SU(N)$, Zeitlin's model space for spherical hydrodynamics. Through explicit examples, we show that our criterion detects conjugate points even in cases where the Ricci curvature is negative along the entire geodesic or the Misio{\l}ek criterion fails; and that it can be used to obtain bounds on the diameter (see Section~\ref{rigid-body}). On $SU(N)$, we obtain that the conjugate locus of the Zeitlin model cannot remain uniformly bounded as $N\to\infty$ (see Section~\ref{zeitlin}).

\subsection{Other results}

In the absence of bounded adjoint action, we distinguish between different types of geodesics and different geometries to obtain our results. We focus on the left-invariant case for concreteness, and the same results can be obtained in the right-invariant case. Throughout the paper, the \emph{isotropy subgroup} $H \subseteq G$ refers to the set of elements $h \in G$ such that the adjoint action $\mathrm{Ad}_h$ is an isometry. 
\begin{itemize}
\item For a Lie group $G$ with a left-invariant metric, and $H$ the isotropy group of $G$ under this metric, we prove that in $G$ and in the homogeneous space $G/H$,  every closed nonhomogeneous geodesic has a conjugate point (Theorem~\ref{closedgeodesicthm} and Corollary~\ref{homogeneousclosed}). In particular,  if $G/H$ has nonpositive sectional curvature, then every closed geodesic must lift to a homogeneous horizontal geodesic in $G$. 

\item As a consequence, we obtain (Corollary~\ref{denseclosedgeodesics}) that if a Lie group $G$ with left-invariant metric has dense closed geodesics, then either it has positive curvature in some section at the identity, or it is abelian and flat. If $H=\mathrm{Iso}(G)$ is the isotropy group of $G$ under its left-invariant metric, and if the homogeneous space $M=G/H$ has dense closed geodesics, then it must have some positive curvature or be naturally reductive. This result also gives an alternative proof of the known fact that the only compact Lie group with nonpositive sectional curvature is the abelian flat torus (see Remark~\ref{torus}). 

\item In the special case of a 3D unimodular Lie group, using the technique of Milnor~\cite{Milnor}, we find a simple expression for the index form based on a natural orthogonal basis, and deduce a sufficient condition for conjugate points to appear along a nonhomogeneous geodesic (Theorem~\ref{3dunimodularthm}). 
This shows existence of conjugate points e.g., in $SL(2, \mathbb{R})$; see Example~\ref{sl2-example}.

\item For homogeneous geodesics, we discuss how to solve the constant-coefficient Jacobi equation (Proposition~\ref{steadytheorem}).  We then show that in semisimple Lie groups and for a steady solution $u_0$ of the Euler--Arnold equation, the corresponding homogeneous geodesic develops conjugate points if $\ad_{u_0}$ has a nice eigenvalue decomposition (in the sense of Theorem \ref{cartan_thm}). This relates conjugate points to linear stability, both in the Eulerian and Lagrangian sense (see Section~\ref{sec:background-lie-groups}), which partially addresses a conjecture raised in \cite{DMSY} and \cite{TauchiYoneda}. As an illustration, we give an example of a left-invariant metric on $SL(3,\mathbb{R})$ for which Ricci is \emph{strictly negative} along a geodesic, but this geodesic still has conjugate points (Example~\ref{sl3-example}).
\end{itemize}

\subsection{Outline of the paper} In Section~\ref{sec:background}, we start by providing background material on left-invariant metrics on Lie groups, homogeneous spaces, and conjugate points in these settings. Section~\ref{sec:compact} contains our main result for Lie groups with bounded adjoint action, and its corollary on compact Lie groups and homogeneous spaces. Section~\ref{sec:noncompact} is devoted to Lie groups where the adjoint action may be unbounded.

\begin{table}[ht!]
\begin{tabular}{r|l}
\hline
$G$ and $H$ & Lie groups with $H\subset G$ \\
$\mathfrak{g}$ and $\mathfrak{h}$ & Lie algebras of $G$ and $H$ \\
$M$ & Riemannian manifold, typically a homogeneous space $G/H$ \\
$\langle\cdot, \cdot\rangle$ & left-invariant positive-definite inner product on $\mathfrak{g}$ \\
$\kappa(\cdot,\cdot)$ & nondegenerate bi-invariant form on $\mathfrak{g}$, not necessarily positive-definite \\  
$\Lambda$ and $\lambda$ & inertia operator giving $\langle\cdot,\cdot\rangle$ in terms of $\kappa$, if $\kappa$ exists, and its eigenvalues \\
$L_g$ and $R_g$ & left- and right-translations by $g\in G$ \\
$\ad_u$ and $\Ad_g$ & standard adjoint actions on $\mathfrak{g}$ \\
$\adstar_u$ and $\Adstar_g$ & transposes of the actions under $\langle\cdot,\cdot\rangle$ \\
$u(t)$ and $u_0$ & solution of the Euler--Arnold equation and its initial condition \\
$\gamma(t)$ & geodesic in $G$ satisfying $\gamma'(t) = DL_{\gamma(t)}u(t)$ \\
$Y(t)$ and $y(t)$ & vector field along a geodesic $\gamma(t)$ and its left-translated version in $\mathfrak{g}$ \\
$I(Y,Y)$ or $I(y,y)$ & index form of a vector field $Y$ along a geodesic, or in terms of $y$ \\
$\unit_{ij}$ & square matrix with $1$ in row $i$ and column $j$, and $0$ elsewhere \\
$\heis_i$ or $\son_{ij}$ & basis vectors for Lie algebras in general, and in Section \ref{rigid-body}\\
$L$ & operator satisfying $[u,v] = L(u\times v)$ on a 3D Lie algebra in Section \ref{milnornonsteadythm} \\
$\delta$, $\xi$, $\ell$ & quantities relevant for 3D unimodular groups defined in Lemma \ref{orthogonallemma} \\
$\adop$ & adjoint operator $\ad_{u_0}$ for a fixed steady $u_0$ in Section \ref{steadygeneral} \\
$\adlambop$ & the operator $\ad_{\Lambda u_0}$, in case $\Lambda$ generates the metric \\
$\eulerop$ &  operator appearing in the linearized Euler equation around a steady $u_0$ \\ 
$\Syl$  & solution of the Sylvester equation for studying homogeneous geodesics\\
\hline
\end{tabular}
\caption{Summary of the main notations.}
\end{table}

\subsection{Acknowledgements:} We thank Renato Ghini Bettiol and Wolfgang Ziller for very helpful discussions during the preparation of this paper. We also thank the referees for their very careful review, and many suggestions for improvement in the exposition.

\newpage 

\section{Background}\label{sec:background}

\subsection{General left-invariant metrics}\label{sec:background-lie-groups}

We begin by recalling standard facts about left-invariant metrics on Lie groups, and introduce conjugate points in this setting (for background material, see \cite{lietext}, \cite{Knapp}, \cite{Helgason}). Note that right-invariant metrics can be handled in exactly the same way with occasional changes of sign that do not affect curvature, so all our results will also hold in that case. Let $G$ be a finite-dimensional Lie group. An inner product $\langle\cdot,\cdot\rangle$ on its Lie algebra $\g$ induces a left-invariant Riemannian metric on $G$ by setting, for any $g \in G$ and $x,y \in T_g G$,
$$\langle x,y\rangle_g := \langle u,v\rangle,\quad \text{where}\quad x=DL_g u, \,\, y= DL_gv, \,\, u,v\in\g,$$
and $L_g$ denotes left-translation by $g\in G$. A curve $\gamma$ in $G$ is a geodesic with respect to this metric if and only if
\begin{equation}\label{eulerarnold}
\gamma'(t) = DL_{\gamma(t)} u(t), \qquad u'(t) = \adstar_{u(t)}u(t),
\end{equation}
where the first equation expresses $u(t)$ as the left-translated \emph{Eulerian} velocity of $\gamma$, and the second is the \emph{Euler--Arnold equation}. The operator $\adstar$ is defined by $\langle \adstar_u v, w \rangle = \langle v, \ad_u w \rangle$ for all $u,v,w \in \mathfrak{g}$. The Euler--Arnold equation admits an equivalent formulation as a conservation law:
\begin{equation}\label{eulerarnold_conservation}
\frac{d}{dt} \left( \Adstar_{\gamma(t)^{-1}} u(t) \right) = 0,
\end{equation}
where $\Adstar_g\colon \mathfrak{g} \to \mathfrak{g}$ is defined by $\langle \Adstar_g u, v \rangle = \langle u, \Ad_g v \rangle$ for all $u,v \in \mathfrak{g}$.
If $\adstar_{u_0} u_0 = 0$, the unique solution of \eqref{eulerarnold} with $u(0) = u_0$ is constant: $u(t) = u_0$ for all $t$. In this case, $u_0$ is called a \emph{steady} solution of the Euler--Arnold equation and the corresponding geodesic $\gamma(t)$, which is a one-parameter subgroup of $G$, is called a \emph{homogeneous geodesic}. If instead $\adstar_{u_0} u_0 \neq 0$, then $u'(t) \ne 0$ for all $t$ by uniqueness of solutions to ODEs, and we call $u(t)$ a \emph{nonsteady} solution and the corresponding geodesic is said to be \emph{nonhomogeneous}.

A \emph{conjugate point} along a geodesic $\gamma$ occurs at time $\tau$ if there exists a nontrivial Jacobi field along $\gamma$ that vanishes at $t=0$ and $t=\tau$. Geometrically, this means that $\gamma$ lies in a one-parameter family of geodesics that all start at $\gamma(0)$ and reconverge (to first order) at $\gamma(\tau)$. In a Lie group, the Jacobi equation reduces to a pair of coupled equations in the Lie algebra obtained by differentiating the geodesic equations \eqref{eulerarnold} with respect to a variation parameter. Thus a Jacobi field along $\gamma$ has the form $J(t) = DL_{\gamma(t)}y(t)$, where $y(t)$ satisfies
\begin{equation}\label{jacobiequation}
y'(t) + \ad_{u(t)}y(t) = z(t), \qquad
z'(t) = \adstar_{u(t)}z(t) + \adstar_{z(t)}u(t).
\end{equation}  
The point $\gamma(\tau)$ is conjugate to $\gamma(0)$ if and only if there exists a nontrivial solution $y$ of \eqref{jacobiequation} such that $y(0) = y(\tau) = 0$. Studying the solutions of the Jacobi equation also leads to a notion of stability of the geodesics: the Eulerian velocity $u(t)$ is called \emph{linearly stable} if every solution $z(t)$ of the second equation of~\eqref{jacobiequation}
remains bounded for all time. It is called linearly stable in the Lagrangian sense if in addition all solutions $y(t)$ of the system with $y(0)=0$ remain bounded for all time. The splitting implies that Lagrangian instability can arise either through Eulerian instability or through a purely Lagrangian mechanism, corresponding to unboundedness of the adjoint action $\Ad_{\gamma(t)}$. As we will see in Theorem~\ref{cartan_thm}, the existence of conjugate points along steady flows is closely related to both notions of stability, as conjectured by Drivas et al.~\cite{DMSY} and Tauchi-Yoneda~\cite{TauchiYoneda} in the infinite-dimensional case.

Conjugate points are also related to curvature: they can only appear in the presence of some positive sectional curvature along the geodesic. This can easily be seen by considering the index form, defined for vector fields $Y(t)$ along $\gamma(t)$ that vanish at $t=0$ and $t=\tau$ by 
\begin{equation}\label{generalindexform}
I(Y,Y) = \int_0^\tau \Big\langle \frac{DY}{dt}, \frac{DY}{dt}\Big\rangle - \big\langle R\big(Y, \gamma'\big)\gamma', Y\big\rangle \, dt,
\end{equation}
where $R$ is the curvature tensor and $\frac{D}{dt}$ is the covariant derivative. A standard fact in Riemannian geometry~\cite{doCarmo} is that $I(Y, Y) \leq 0$ for some $Y$, not identically zero, vanishing at $t = 0$ and $t = \tau$ if and only if there exists a nontrivial Jacobi field $J(t)$ which vanishes at time $t=0$ and at some time $t_0$ with $0<t_0\leq \tau$. In particular, if the sectional curvature is everywhere nonpositive, then the integrand in \eqref{generalindexform} is nonnegative and no conjugate points can occur. In a Lie group with left-invariant metric, the index form can be simplified significantly, as we will show in Lemma \ref{lemma_index}.

\subsection{Homogeneous spaces}\label{homogeneoussection}
 
Given a Lie group $G$ and a compact subgroup $H$, the quotient space $M=G/H$ is the set of left cosets. If $G$ is equipped with a left-invariant metric, then we define $\hperp$ to be the orthogonal complement of the Lie subalgebra $\mathfrak{h}\subset \mathfrak{g}$, and $\hperp$ may be identified with the tangent space $T_{[H]}M$, where $[H]$ denotes the coset of the identity element in $M$. We consider the following notions of compatibility between this subgroup and the metric:

\setcounter{eq_counter_save}{\value{equation}}
\begingroup
\renewcommand\theequation{C\arabic{equation}}
\setcounter{equation}{0}
\begin{align}
    \forall v\in \mathfrak{h}, \ad_v \text{ maps $\hperp$ to itself and is antisymmetric on $\hperp$.} \label{condition1} \\
    \forall v\in\mathfrak{h}, \ad_v \text{ is antisymmetric on  $\mathfrak{g}$.} \label{condition2} \\ 
    \forall u\in\mathfrak{g}, \ad_u \text{ is antisymmetric on  $\mathfrak{g}$.} \label{condition3} 
\end{align}
\endgroup
\setcounter{equation}{\value{eq_counter_save}}

Condition~\eqref{condition1} is the necessary and sufficient condition for the left-invariant metric to descend to a Riemannian metric on $M$, which makes the projection $\pi\colon G\to M$ a Riemannian submersion, and the left action of the group $H$ on $M$ will consist of isometries. (Note that since $M$ is the set of left cosets, the right action of $H$ on $M$ is trivial, but the left action is not.)
In this case we have a \emph{Riemannian homogeneous space}. 

Condition~\eqref{condition2} implies Condition~\eqref{condition1}, since $\ad_v$ maps $\mathfrak{h}$ to itself, and antisymmetry of $\ad_v$ on all of $\mathfrak{g}$ implies it must also map $\hperp$ to itself. Condition~\eqref{condition2} is equivalent to saying that both left- and right-translations by elements of $H$ are isometries of the metric on $G$, and in particular the inherited metric on $H$ itself is actually bi-invariant. In most examples we have in mind, such as the Zeitlin model on $G=SU(N)$ with $H=SO(3)$, this condition is satisfied. Note that if only the desired Riemannian metric on $M$ is known, with $G$ acting on $M$ by left-translations that are isometries, we may reconstruct a left-invariant metric on $G$ that gives this Riemannian submersion structure by specifying an arbitrary left-invariant metric on $H$, and defining the inner product on $\mathfrak{g}$ to be the direct sum of the inner products on $\mathfrak{h}$ and $T_{[H]}M$. Any such choice will lead to the same metric on $M$. Thus since $H$ is compact, we may as well use a bi-invariant metric on $H$ to construct the metric on $G$. In particular if we start with a homogeneous Riemannian metric on $M$ and have freedom to choose the metric on $G$, we can assume it satisfies \eqref{condition2}.

Condition \eqref{condition3} is equivalent to saying the metric on $G$ is bi-invariant, and implies the previous two conditions. The resulting homogeneous space is called a \emph{normal homogeneous space} or \emph{standard homogeneous space}. A Lie group which admits a bi-invariant metric must be the direct product of a compact group and an abelian group (see \cite{Milnor}), and for a compact semisimple group, the bi-invariant metric is unique up to constant factors. Hence the Riemannian geometry of a normal Riemannian homogeneous space arising from a compact semisimple group $G$ is completely determined by $G$ and its subgroup $H$. This is too restrictive an assumption for us, since we are most interested in geometries arising from physics, such as the Zeitlin model for ideal fluids and the motion of rigid bodies in $\mathbb{R}^N$ with possibly unequal axes of inertia. 
See Chavel~\cite{Chavel2} for a description of the conjugate points on normal homogeneous spaces. 

Another commonly considered condition is that $M=G/H$ is \emph{naturally reductive}, which is the statement that, in addition to \eqref{condition1}, the map $P_{\hperp}  \circ \,\ad_w\colon \hperp\to\hperp$ is antisymmetric for every $w\in\hperp$, where $P_{\hperp}$ denotes orthogonal projection onto $\hperp$. Conjugate points in this case were considered in detail by Ziller~\cite{Ziller2}; in this case the Jacobi equation becomes a constant-coefficient system, which simplifies things considerably.

In what follows, we assume only \eqref{condition1}, except in Corollary~\ref{homogeneousclosed}. The index form on a Riemannian homogeneous space can be simplified using O'Neill's formula (originally proved for any Riemannian submersion). Note that any vector field along a curve in the group $G$ can be viewed as the restriction of a time-dependent left-invariant vector field on the entire group. 
\begin{theorem}[O'Neill~\cite{oneillconjugate}]\label{oneillhomogeneousindex}
Suppose $M=G/H$ is a Riemannian homogeneous space and let $\gamma\colon [0,\tau]\to G$ be a horizontal geodesic with time-dependent Eulerian velocity field $u(t)$. Then for any time-dependent left-invariant vector field $Y=\verproj Y+\horproj Y$ along $\gamma$ with horizontal part $\horproj Y$ vanishing at the endpoints, let $Y_*=\pi_*(Y)=\pi_*(\horproj Y)$ be the corresponding vector field along $\pi\circ\gamma$ in $M$, noting that $Y_*$ vanishes at the endpoints.  Then the index forms of $G$ and $M$ are related by 
\begin{equation}\label{eq_oneill}
    I_G(Y,Y) = I_M(Y_*,Y_*) + \int_0^{\tau} \lvert \verproj Y'+\verproj[u,\horproj Y]\rvert^2 \, dt.
\end{equation}
\end{theorem}
The version of formula \eqref{eq_oneill} given in \cite{oneillconjugate} differs only in the squared term, which can be simplified by checking that \eqref{condition1} implies $\adstar_\h\h^\perp\subset\h^\perp$ and $\adstar_{\h^\perp} \h\subset\h^\perp$. We can use this theorem to show that conjugate points along horizontal geodesics on the group $G$ descend to the quotient space $G/H$.

\begin{corollary}\label{homogeneousconjugatepointscor}
    Let $M=G/H$ be a Riemannian homogeneous space and  $\gamma\colon [0,\tau]\to G$ be a geodesic with $\gamma'(0)$ horizontal. Suppose there is a conjugate point along $\gamma$ in $G$ at time $\tau$. Then the geodesic $\pi\circ\gamma$ in $M$ also has a conjugate point at some time $\tau'\le \tau$. 
\end{corollary}

The proof follows by projecting a nontrivial Jacobi field $J$ along the lifted geodesic in $G$ to the base space $M = G/H$, and using the Jacobi equations \eqref{jacobiequation} to see that if $\pi_*J = 0$ (so that $J$ is purely vertical) then $J$ must be trivial, a contradiction.

\subsection{Known criteria for the existence of conjugate points}\label{sec:background-conjugate-criteria}

Since our goal in this paper is to present new criteria for the existence of conjugate points, here we revisit two well-known methods: the Gromoll–Meyer criterion and the Misio{\l}ek criterion.  

The \emph{Gromoll--Meyer criterion} states that if Ricci curvature is always nonnegative along a geodesic $\gamma(t)$ and strictly positive for at least some time $t = t_0$, then $\gamma$ eventually develops conjugate points (see \cite{GM}). This is true on any (complete) Riemannian manifold $M$. In the sequel, we establish new criteria that do not rely on curvature assumptions. In fact, in many of our examples, Ricci curvature can take on negative values in some directions, or even be negative along the geodesic containing conjugate points (see Section~\ref{rigid-body} and Example~\ref{sl3-example}).

The \emph{Misio{\l}ek criterion}, which first arose in the context of geometric hydrodynamics \cite{Benn, DMSY,M1,P1}, can be readily generalized to any Lie group $G$ as follows~\cite{tauchiyonedapositivity}.
\begin{proposition}[Misio{\l}ek criterion]
Let $\gamma(t)$ be a homogeneous geodesic with initial velocity $u_0\in\g$. If 
\begin{equation}\label{eq_m_curvature}
\big\langle\ad_vu_0 + \adstar_vu_0, \ad_vu_0\big\rangle < 0 \qquad \text{for some $v\in\mathfrak{g}$},
\end{equation}
then $\gamma(t)$ eventually develops conjugate points. 
\end{proposition} 

Notice that the Misio{\l}ek criterion \eqref{eq_m_curvature} is also a condition to get positive curvature on a plane containing $u_0$. This can be seen from Arnold's curvature formula (see \cite{AK}), reformulated as in \cite{leepreston},
\begin{equation*} 
\big\langle R(u,v)v,u\big\rangle = \frac{1}{4} \lvert\adstar_uv+\adstar_vu+\ad_vu\rvert^2 
    - \langle\ad_vu + \adstar_vu, \ad_vu\rangle - \langle\adstar_uu, \adstar_vv\rangle.
\end{equation*}  
Using the fact that $u_0$ is a steady solution of \eqref{eulerarnold}, so that $\adstar_{u_0}u_0=0$, we see that \eqref{eq_m_curvature} directly implies positive curvature on the $2$-plane spanned by $u_0$ and $v$. We shall later see that the criteria developed here can detect conjugate points in cases where the Misio{\l}ek criterion fails (see Remark~\ref{rem-Misiolek}).

\section{Lie groups with bounded adjoint action}\label{sec:compact}

\subsection{Main theorem}

We begin this section with a lemma showing how the index form can be simplified on a general Lie group. Following that, we present our main theorem and derive some consequences for Lie groups that admit a bi-invariant metric. An important consequence is that conjugate points occur for every left- or right-invariant metric on any nonabelian compact Lie group, and along every geodesic if the group is semisimple as well. The special cases of  
$SO(N)$ and $SU(N)$ 
are of particular interest, since these are configuration spaces for rigid bodies and fluids.

\begin{lemma}\label{lemma_index}
Let $\gamma(t)$ be a geodesic on a Lie group $G$ with a left-invariant metric, and let $Y(t)$ be a vector field along $\gamma(t)$ vanishing at $t=0$ and $t=\tau$. Let $u(t)=DL_{\gamma(t)^{-1}}\gamma'(t)$ and $y(t)=DL_{\gamma(t)^{-1}}Y(t)$ be the left translations of $\gamma'$ and $Y$ to the Lie algebra. Then the index form is given by
    \begin{equation}\label{indexform}
    I(Y,Y) = \int_0^{\tau} \big\langle z(t),z(t)\rangle - \big\langle [y(t),z(t)], u(t)\big\rangle \, dt, \qquad z(t):=y'(t) + [u(t),y(t)].
    \end{equation}
On the other hand if $v(t) = DR_{\gamma(t)^{-1}}Y(t)$ is the right translation, the index form is given by    \begin{equation}\label{indexformconservation}
    I(Y,Y) = \int_0^{\tau} \big\langle \Ad_{\gamma(t)^{-1}}v'(t),\Ad_{\gamma(t)^{-1}}v'(t)\big\rangle - \big\langle [v(t),v'(t)], u_0\big\rangle \, dt.
    \end{equation}    
\end{lemma}

\begin{proof}
Working with local left-invariant extensions, we can write
\begin{equation*}
\frac{D Y}{dt} = DL_{\gamma(t)}\big( y'(t) + \nabla_{u(t)} y(t) \big),
\end{equation*}
where $\nabla$ denotes the Levi-Civita connection. Since the Riemann curvature tensor is also left-invariant, we have $\langle R(Y,\gamma')\gamma', Y\rangle = \langle R(y,u)u, y\rangle$, so that the index form~\eqref{generalindexform} becomes 
    \begin{align*}
        I(Y,Y) &= \int_0^{\tau} \langle y'+\nabla_uy, y'+\nabla_uy\rangle - \langle \nabla_y\nabla_uu - \nabla_u\nabla_yu + \nabla_{[u,y]}u, y\rangle \, dt \\
         &= \int_0^{\tau} \langle z+\nabla_yu, z+\nabla_yu\rangle - \langle \nabla_y\nabla_uu, y\rangle - \langle \nabla_yu, \nabla_uy\rangle - \langle \nabla_{[u,y]}u, y\rangle \, dt \\
         &= \int_0^{\tau} \lvert z\rvert^2  +2 \langle z, \nabla_yu\rangle -
         \langle \nabla_yu, [u,y]\rangle -
         \langle \nabla_y\nabla_uu, y\rangle  - \langle \nabla_{[u,y]}u, y\rangle \, dt.
    \end{align*}
    Since $u(t)$ satisfies the Euler--Arnold equation \eqref{eulerarnold}, which is equivalent to $u' + \nabla_uu = 0$, an integration by parts gives 
    \begin{align*}
        I(Y,Y) &= 
    \int_0^{\tau} \lvert z\rvert^2  +2 \langle z, \nabla_yu\rangle -
         \langle \nabla_yu, [u,y]\rangle -
         \langle \nabla_{y'}u, y\rangle -
          \langle \nabla_yu, y'\rangle - \langle \nabla_{[u,y]}u, y\rangle \, dt \\ 
         &=     \int_0^{\tau} \lvert z\rvert^2  +  \langle z, \nabla_yu\rangle  -
         \langle \nabla_zu, y\rangle \, dt \\
         &=   \int_0^{\tau} \lvert z\rvert^2  - \langle [y,z], u\rangle \, dt,
    \end{align*}
    as desired.

Now suppose $v(t)$ is the right translated Jacobi field. Then $v(t) = \Ad_{\gamma(t)}y(t)$ and $$ v'(t) = \Ad_{\gamma(t)}\big( y'(t) + \ad_{u(t)}y(t)\big) = \Ad_{\gamma(t)}z(t).$$
Thus the first term in \eqref{indexform} becomes
$$\lvert z(t)\rvert^2 = \big\lvert \Ad_{\gamma(t)^{-1}} v'(t)\big\rvert^2. $$
To simplify the other term, we use the conservation law     \eqref{eulerarnold_conservation} to obtain 
 \begin{align*}
     \big\langle [y(t),z(t)], u(t)\big\rangle &= \big\langle [\Ad_{\gamma(t)^{-1}}v(t),\Ad_{\gamma(t)^{-1}} v'(t)], \Adstar_{\gamma(t)} u_0\big\rangle \\
     &= \big\langle \Ad_{\gamma(t)^{-1}} [v(t), v'(t)], \Adstar_{\gamma(t)} u_0\big\rangle \\
     &= \big\langle [v(t),v'(t)], u_0\big\rangle.
 \end{align*} 
 Putting these together gives \eqref{indexformconservation}.
\end{proof}

\begin{theorem}\label{generaltheorem}
Let $G$ be a Lie group with left invariant metric, $u_0\in\g$, and $\gamma(t)$ be the geodesic in $G$ with initial conditions $\gamma(0)=\id$ and $\gamma'(0)=u_0$. If $u_0$ is not orthogonal to $[\mathfrak{g},\mathfrak{g}]$ and $\Ad_{\gamma(t)}$ is uniformly bounded below as an operator on $\mathfrak{g}$, then there is a conjugate point along the geodesic $\gamma$. If $u_0$ is orthogonal to $[\mathfrak{g},\mathfrak{g}]$, then there is no conjugate point along $\gamma$. 
\end{theorem}

\begin{proof} 
    By Lemma \ref{lemma_index}, we can write the index form for $Y(t) = DL_{\gamma(t)} y(t) = DR_{\gamma(t)}v(t)$ as 
    $$ I(Y,Y) = \int_0^{\tau} \big\lvert  \Ad_{\gamma(t)^{-1}} v'(t)\big\rvert^2 - \langle u_0, \ad_{v(t)}v'(t)\rangle \, dt. $$
If $u_0$ is orthogonal to $[\mathfrak{g},\mathfrak{g}]$, then obviously the second term vanishes, and the index form is positive-definite for any $\tau>0$. Hence there are no conjugate points.

Now assume that $\Ad_{\gamma(t)}$ is bounded from below and $u_0$ is not orthogonal to $[\mathfrak{g},\mathfrak{g}]$. 
By assumption there is a $\rho>0$ such that $\langle \Ad_{\gamma(t)^{-1}}w,\Ad_{\gamma(t)^{-1}}w\rangle \le \rho \langle w,w\rangle$ for all $t\in\mathbb{R}$ and all $w\in\mathfrak{g}$. 
We then obtain for any variation field that
        $$ I(Y,Y) \le  \int_0^{\tau} \rho \langle v'(t), v'(t)\rangle - \langle u_0, \ad_{v(t)}v'(t)\rangle \, dt. $$
    Since $u_0$ is not orthogonal to $[\mathfrak{g},\mathfrak{g}]$, there are vectors $p,q\in\mathfrak{g}$ such that $\epsilon := \langle [p,q], u_0\rangle$ is positive; we may assume without loss of generality that $p$ and $q$ are orthonormal. Define a variation field by 
    $$ v(t) = f(t) p + g(t)q $$
    where $f(0)=f(\tau)=g(0)=g(\tau)=0$. Then $$\langle u_0,\ad_{v(t)}v'(t)\rangle 
    = \epsilon \big(f(t)g'(t)-g(t)f'(t)\big).$$ 
    Thus the index form satisfies 
    $$ I(Y,Y) \le \int_0^{\tau} \rho \big(f'(t)^2 + g'(t)^2\big) - \epsilon \big( f(t)g'(t)-g(t)f'(t)\big) \, dt.$$
    Choosing $\tau = 2\pi\rho/\epsilon$ with 
    $$ f(t) = 
     \sin{\left( \frac{\epsilon t}{2\rho}\right)}\cos{\left( \frac{\epsilon t}{2\rho}\right)}
      \qquad \text{and}\qquad 
    g(t) = \sin^2{\left( \frac{\epsilon t}{2\rho}\right)}, $$
    we easily compute $I(Y,Y)\le 0$. 
\end{proof}

\begin{remark}
Note that the geodesics that are shown to be free of conjugate points in Theorem~\ref{generaltheorem} are necessarily homogeneous, since $u_0$ orthogonal to $[\mathfrak{g}, \mathfrak{g}]$ implies $\adstar_vu_0=0$ for all $v\in\g$, and $v=u_0$ in particular. Milnor proved in \cite{Milnor} that if $u_0$ is orthogonal to the commutator ideal $[\mathfrak{g}, \mathfrak{g}]$, then $\mathrm{Ric}(u_0, u_0) \leq 0$. The second part of Theorem~\ref{generaltheorem} provides an alternative proof of this fact, since $\mathrm{Ric}(u_0, u_0) > 0$ would imply the existence of conjugate points along the homogeneous geodesic by Gromoll--Meyer's criterion (see Section~\ref{sec:background-conjugate-criteria}). Note that Herring~\cite{herring} also proved that if $u_0$ happens to be \emph{steady}, then this orthogonality condition guarantees no conjugate points along the homogeneous geodesic, but our theorem holds also in the nonsteady case.
\end{remark}

\begin{example}\label{ex_no_conjugate}
To see that the assumption of a uniform lower bound for $\mathrm{Ad}_{\gamma(t)}$ is necessary in Theorem \ref{generaltheorem}, consider the closed subgroup of $GL(4, \mathbb{R})$ given by matrices of the form
\begin{equation*}%\label{def_L4}
G = \left\{ 
\left(\begin{matrix}
1 & a & c+ab & d \\
0 & 1 & b & -c \\
0 & 0 & 1 & a \\
0 & 0 & 0 & 1 
\end{matrix}\right) \, : \, a, b, c, d \in \mathbb{R} \right\}.
\end{equation*}
Its Lie algebra $\mathfrak{g}$ is nilpotent and it is spanned by 
\begin{gather*}
\begin{split}
\heis_1 = \unit_{12} + \unit_{34},\quad 
\heis_2 = \unit_{23},\quad
\heis_3 = \unit_{13} - \unit_{24}, \quad
\heis_4 = -2\unit_{14}, 
\end{split}
\end{gather*}
where $\unit_{ij}$ is the matrix having $1$ in position $(i, j)$ and zeros everywhere else. These satisfy $[\heis_1, \heis_2] = \heis_3$, $[\heis_1, \heis_3] = \heis_4$ and all other brackets are zero. We consider the left-invariant metric that comes from declaring the $\heis_i$ to be an orthonormal basis and let $u_0 = \heis_2 - \heis_4$. One verifies that the unique geodesic $\gamma$ in this metric with $\gamma(0) = \mathrm{Id}$ and $\gamma'(0) = u_0$ is homogeneous and satisfies $\ad_{u_0}=-\unit_{31}$ as an operator in the $4$-dimensional Lie algebra, so that $\Ad_{\gamma(t)} = \Id - t \unit_{31}$. It has no conjugate points, as one can see from the index form, which in this case is 
\begin{align*}
I(Y, Y) &= \int_0^{\tau} \big(v_1'\big)^2 + \big(tv_1'+v_3'\big)^2 + 2v_1v_3' + \big(v_2'\big)^2 + \big(v_4'\big)^2\,dt \\ 
&= \int_0^{\tau} \big(v_1'\big)^2 + (w')^2 + \big(v_2'\big)^2 + \big(v_4'\big)^2\,dt,
\end{align*}
where $v(t) = DR_{\gamma(t)}^{-1}Y(t) =  \sum\limits_{i = 1}^4 v_i(t)\heis_i$ and $w = tv_1 + v_3$. This is clearly positive-definite, so we cannot have conjugate points. At the same time, the smallest singular value of $\mathrm{Ad}_{\gamma(t)} = \mathrm{Id} - t\unit_{31}$ goes to zero as $t$ increases, so $\mathrm{Ad}_{\gamma(t)}$ is not bounded from below.
\end{example}

Our Main Theorem from the Introduction is now a consequence of Theorem \ref{generaltheorem}.

\begin{corollary}\label{easyconjugatethm}
Suppose $G$ is a Lie group with a positive-definite bi-invariant metric $\bim(\cdot,\cdot)$, and that a left-invariant metric $\langle \cdot,\cdot\rangle$ is defined in terms of a positive-definite symmetric inertia operator $\Lambda\colon\mathfrak{g}\to\mathfrak{g}$, with $\langle u,v\rangle = \bim(u,\Lambda v) \; \forall u,v\in\mathfrak{g}$. Then for any initial velocity $u_0 \in \mathfrak{g}$, the geodesic $\gamma(t)$ of the left-invariant metric $\langle\cdot,\cdot\rangle$ with $\gamma(0) = \Id$ and $\gamma'(0) = u_{0}$ has a conjugate point if and only if $\Lambda u_0$ is not in the center of $\mathfrak{g}$. Furthermore, if $\gamma$ has a conjugate point then the first conjugate time $\tau$ satisfies
$$\frac{2\pi \lambda_m}{\lVert \ad_{\Lambda u_0}\rVert_{op}} \leq  \tau \leq \frac{2\pi \lambda_M}{\lVert \ad_{\Lambda u_0}\rVert_{op}},$$
where $\lambda_m$ and $\lambda_M$ are the smallest and largest eigenvalues of $\Lambda$, respectively, and the operator norms are computed using $\bim(\cdot, \cdot)$.
\end{corollary}

\begin{proof}
By Lemma \ref{lemma_index}, the index form for $Y(t)=DR_{\gamma(t)}v(t)$ can be written in terms of the bi-invariant metric $\bim(\cdot, \cdot)$ as
\begin{align*} I(Y,Y) &= 
\int_0^{\tau} \big\langle \Ad_{\gamma(t)^{-1}}v'(t), \Ad_{\gamma(t)^{-1}}v'(t)\big\rangle-\bim\big([v(t),v'(t)],\Lambda u_0\big)\, dt \\
&=\int_0^{\tau} \bim\big(\Ad_{\gamma(t)^{-1}}v'(t),\Lambda \Ad_{\gamma(t)^{-1}}v'(t)\big)-\bim\big(\ad_{\Lambda u_0}v(t),v'(t)\big)\, dt 
\end{align*}
To estimate the first term in the index form, note that  
$$\lambda_m\bim(w, w) \leq \bim(w,\Lambda w) \le \lambda_M \bim(w,w), \qquad \text{for all } w \in \mathfrak{g}.$$
These imply that
    $$ 
    \lambda_m \bim\big( v'(t), v'(t) \big)
    \le
    \bim\big(\Ad_{\gamma(t)^{-1}}v'(t),\Lambda \Ad_{\gamma(t)^{-1}} v'(t))
    \le \lambda_M \bim\big( v'(t), v'(t) \big) $$
    since the metric $\bim(\cdot,\cdot)$ is bi-invariant.  
  Now, note that the operator $\adlambop := \ad_{\Lambda u_0}$ is antisymmetric on $\mathfrak{g}$ in the bi-invariant metric, so we have a decomposition
\begin{equation}\label{eq_g_decomp}
    \mathfrak{g} = \ker \adlambop \oplus V_1 \oplus \cdots \oplus V_k
    \end{equation}
    such that each subspace $V_i$ is spanned by a pair of orthonormal (in the bi-invariant metric) vectors $p_i$, $q_i$, with $\adlambop p_i = \alpha_i q_i$ and $\adlambop q_i = -\alpha_i p_i$ for some $\alpha_i > 0$.

    If $\Lambda u_0$ is a central element, then $\mathfrak{g} = \ker \adlambop$. In this case $u_0$ is orthogonal to $[\mathfrak{g},\mathfrak{g}]$,  
    and by Theorem~\ref{generaltheorem} there are no conjugate points along $\gamma$.

    If $\Lambda u_0$ is not a central element, then we can construct a nontrivial variation within each $V_i$ as follows: let $\omega := \frac{\alpha_i}{2\lambda_M}$, and define a variation field by 
    $$ v(t) = \sin{(\omega t)} \big(\cos{(\omega t)} p_i + \sin{(\omega t)} q_i\big),$$
    which vanishes at $t=0$ and $t=\pi/\omega$. Then it is easy to compute that the index of $v$ is 
    \begin{align*}
        I(v,v) &\le \int_0^{\pi/\omega} \lambda_M \bim\big(v'(t),v'(t)\big) - \bim\big(v'(t),\ad_{\Lambda u_0}v(t)\big) \, dt \\
        &=         \int_0^{\pi/\omega} \big( \lambda_M \omega^2 -  2\lambda_M \omega^2 \sin^2{(\omega t)} \big)\, dt = 0.
    \end{align*}
   Hence there is a conjugate point along $\gamma$ that happens no later than time $\pi/\omega = 2\pi\lambda_M/\alpha_i$, and the earliest such time is obtained by taking the largest $\alpha_i$, which equals $\| \ad_{\Lambda u_0} \|_{op}$.

Conversely, to obtain a lower bound on the first conjugate time, consider an arbitrary variation field $v(t)$ vanishing at $t = 0$ and $t = \tau$. Decompose $v$ according to \eqref{eq_g_decomp} as $v(t) = v_0(t) + \cdots + v_k(t)$ and write 
\begin{equation*}
v_i(t) = f_i(\ell t)p_i + g_i(\ell t)q_i, \qquad i = 1, \ldots, k,
\end{equation*}
for some functions $f_i$ and $g_i$ vanishing at $t = 0$ and $t = \ell \tau$, for $\ell:=\max\{\alpha_i\}/\lambda_m$. Then by orthogonality, we get:
\begin{equation}\label{index_estimate} 
I(v, v) \ge \sum\limits_{i = 1}^k  \int_0^{\ell \tau} \ell \lambda_m (f_i'^2+g_i'^2) + 2 \alpha_i f_i'g_i \, dt \ge \sum\limits_{i = 1}^k \alpha_i \int_0^{\ell \tau}  (f_i'^2+g_i'^2) + 2 f_i'g_i \, dt.
\end{equation}
This latter expression is precisely a multiple of the index form on $SO(3)$ under the standard round metric (corresponding to a rigid body with all moments of inertia equal), where we know the first conjugate point happens at $2\pi$. Hence for $\ell\tau\le 2\pi$ we see that every summand in \eqref{index_estimate} is nonnegative, and the first conjugate point cannot happen before $2\pi/\ell$.
\end{proof}

Note that this reproduces the known result for conjugate points under the bi-invariant metric, where $\Lambda$ is the identity, so the result is sharp.

\begin{remark}
The decomposition of the Lie algebra into $2$-planes which are invariant under $\ad_{\Lambda u_0}$ suggests that conjugate points on compact Lie groups arise in pairs, and therefore have even multiplicity. We do not know if this is true in general, but it is certainly an interesting question with potential implications for fluid dynamics (see \cite{DMSY}). 
\end{remark}
 
In the important special case where $G$ is semisimple and compact, we get a conjugate point along every geodesic in $G$ and every homogeneous space $M=G/H$. We emphasize again that the Riemannian metric on $M$ need not arise from the bi-invariant metric on $G$.

\begin{corollary}\label{homogeneouscorollary} 
If $G$ is a compact semisimple Lie group with any left-invariant metric and $H$ is any closed subgroup, then every geodesic in the homogeneous space $G/H$ develops a conjugate point. In particular if $M$ is a compact Riemannian homogeneous space with $\pi_1(M)$ finite, then every geodesic develops a conjugate point.
\end{corollary}

\begin{proof}
    Since $G$ is compact, it has a bi-invariant metric. Moreover, since $G$ is semisimple, the center of its Lie algebra is trivial, so every geodesic in $G$ has a conjugate point by Corollary \ref{easyconjugatethm}.
    O'Neill's result (Corollary ~\ref{homogeneousconjugatepointscor}) implies that every geodesic on the homogeneous space $G/H$ has a conjugate point no later than the corresponding horizontal geodesic on $G$.

    A compact Riemannian homogeneous space $M$ can be written as the quotient of $G/H$ where $G$ is also compact. The Lie algebra $\mathfrak{g}$ can be written as a direct sum of an abelian part and a semisimple part (Knapp~\cite{Knapp} Chapter IV), and $\pi_1(M)$ is infinite if and only if this abelian part is nontrivial. Thus if $\pi_1(M)$ is finite, then $G$ is semisimple. 
\end{proof}

In the following subsections, we present applications of Corollary~\ref{easyconjugatethm} to Lie groups with invariant metrics which are of interest from a physical standpoint, namely for rigid bodies and fluids.

\subsection{Rotations of rigid bodies}\label{rigid-body}

Consider the group of rotations $SO(N)$. Any left-invariant metric on $SO(N)$ can be obtained from the canonical bi-invariant metric $\kappa(\cdot,\cdot)$ and a symmetric positive definite operator $\Lambda:\so(N)\rightarrow\so(N)$, i.e.
\begin{equation}\label{son-metric}
\langle u,v\rangle:=\bim(u,\Lambda v),\quad\text{where}\quad \bim(u, v) = \tfrac{1}{2}\mathrm{Tr}(uv^\top), \qquad u,v\in\mathfrak{so}(N),
\end{equation}
and $\mathrm{Tr}$ denotes the trace. An important special case is the \emph{generalized rigid body metric} obtained by considering $\Lambda(u) := \frac{1}{2}(Mu + uM)$, where $M$ is a symmetric matrix with positive eigenvalues $\mu_1, \hdots, \mu_N$ (cf. \cite{Frahm}, \cite{Ratiu}, \cite{FM}, \cite{Manakov}). This is the kinetic energy metric describing the rotations of an $N$-dimensional free rigid body with principal moments of inertia $\mu_1,\hdots,\mu_N$. For $1\leq i<j\leq N$, let $\son_{ij}=\unit_{ji}-\unit_{ij}$ denote the matrix full of zeros except for $-1$ in position $(i,j)$ and $1$ in position $(j,i)$. Then $\son_{ij}$ is an eigenvector of $\Lambda$ for the eigenvalue $\frac{1}{2}(\mu_i+\mu_j)$, and $\{\son_{ij}\}_{1\leq i<j\leq N}$ is a so-called \emph{nice basis} (see \cite{Krishnan}) of the Lie algebra $\so(N)$. In this basis, the Ricci tensor is diagonal, with diagonal elements
$$\mathrm{Ric}(\son_{ij},\son_{ij})=\sum_{k\ne i,j}\frac{2\mu_i\mu_j}{(\mu_i+\mu_k)(\mu_j+\mu_k)}.$$
Thus the Ricci curvature of the generalized rigid body metric on $SO(N)$ is everywhere positive, generalizing a result of Milnor~\cite{Milnor} when $N=3$, and conjugate points exist along any geodesic by the Gromoll--Meyer criterion \cite{GM}. 

Corollary~\ref{easyconjugatethm} shows that conjugate points exist in fact for \emph{any} left-invariant metric on $SO(N)$ and along any geodesic, even in directions of negative Ricci curvature. For example, relaxing the positivity constraint on $\mu_1,\hdots,\mu_N$ in the generalized rigid body metric and setting $\mu_1=-1$, $\mu_i=2$ for $i=2,\hdots,N$, one gets negative Ricci curvature in the direction of $\son_{12}$: $\mathrm{Ric}(\son_{12}, \son_{12})<0$. The geodesic $\gamma$ with initial conditions $\gamma(0)=\id$ and $\gamma'(0)=u_0:=\son_{12}$ is steady since $\adstar_{u_0}u_0=-\Lambda^{-1}(\ad_{u_0}\Lambda u_0)=-\Lambda^{-1}(\frac{1}{2}\ad_{u_0}u_0)=0$, and $\mathrm{Ric}(u(t), u(t))<0$ for all $t$. Corollary~\ref{easyconjugatethm} shows that conjugate points exist along this geodesic despite the negative Ricci curvature. More generally, if the operator $\Lambda$ in \eqref{son-metric} is chosen to have the $\son_{ij}$'s as eigenvectors for \emph{arbitrary} positive eigenvalues $\lambda_{ij}$, then (setting $\lambda_{ji}=\lambda_{ij}$ for convenience) the diagonal Ricci tensor is given by
$$\mathrm{Ric}(\son_{ij},\son_{ij})=\sum_{k\ne i,j} \frac{(\lambda_{ij}-\lambda_{ik}+\lambda_{jk})(\lambda_{ij}+\lambda_{ik}-\lambda_{jk})}{2\lambda_{ik}\lambda_{jk}},$$
which can obviously be negative for some choices of positive $\lambda_{ij}$.

\begin{remark}\label{rem-Misiolek}
If $u_0$ is a steady solution of the Euler equation corresponding to eigenvalue $\lambda$ of $\Lambda$, the Misio{\l}ek criterion \eqref{eq_m_curvature} is written in terms of $\adop=\ad_{u_0}$ as  
$$\langle\ad_{v}u_0+\adstar_{v}u_0,\ad_vu_0\rangle=\bim(\Lambda\ad_{v}u_0-\lambda\ad_vu_0,\ad_vu_0)=\bim( (\Lambda -\lambda I)\adop v, \adop v).$$
Thus, we see that for the steady solution corresponding to, e.g., the smallest eigenvalue $\lambda$ of $\Lambda$, the Misio{\l}ek criterion fails to detect the existence of conjugate points on the corresponding homogeneous geodesic, even though here we know every geodesic has them.  
\end{remark}

In addition, Corollary~\ref{easyconjugatethm} provides an upper bound on the diameter of $SO(N)$ under the left-invariant metric~\eqref{son-metric}. For any $u_0\in\mathfrak{so}(N)$, denoting $\tau(u_0)$ the first conjugate time along the geodesic $\gamma$ with initial conditions $\gamma(0)=\id$ and $\gamma'(0)=u_0$, the diameter is given by $\mathrm{Diam}(SO(N),\Lambda)=\sup_{\lvert u_0\rvert =1}\tau(u_0)$ where $\lvert u_0\rvert^2=\kappa(u_0,\Lambda u_0)$, and verifies the following bound
\begin{equation}\label{diameter_son}
\mathrm{Diam}(SO(N),\Lambda)\leq \frac{2\pi\lambda_M}{\inf_{\lvert u_0\rvert=1}\|\ad_{\Lambda u_0}\|_{op}}
\end{equation}
where $\lambda_M$ is the largest eigenvalue of $\Lambda$.
For $N\ge 3$ and a fixed element $u_0\in\mathfrak{so}(N)$, $m_0:=\Lambda u_0\in\mathfrak{so}(N)$ is an antisymmetric real $N\times N$ matrix with a basis of $N$ complex eigenvectors $x_k$ with pure imaginary eigenvalues $ i\alpha_k$ for $\alpha_k\ge 0$ and $1\le k\le p:=\lfloor N/2\rfloor$, together with $x_{-k}:=\overline{x_k}$ having eigenvalues $i\alpha_{-k}:=-i\alpha_k$, with $x_0$ and $\alpha_0=0$ in case $N$ is odd. The eigenvectors of $\ad_{m_0}$ are given by $X_{jk} = x_jx_k^\top - x_kx_j^\top$ with corresponding eigenvalues $i(\alpha_j+\alpha_k)$.  
Thus if the eigenvalues $i\alpha_k$ are ordered with $\alpha_k\le \alpha_{k+1}$, then we see that the operator norm is $\lVert \ad_{m_0}\rVert_{op} = \alpha_{p-1}+\alpha_p$. 
If $\lambda_M$ and $\lambda_m$ are respectively the largest and the smallest eigenvalues of $\Lambda$, the unit-speed constraint $\kappa(u_0,\Lambda u_0) = 1$ implies $\lambda_m \le  \kappa(m_0,m_0) \le \lambda_M$. Thus, the upper bound in \eqref{diameter_son} is found by minimizing $\alpha_{p-1}+\alpha_p$ subject to the constraints $0\le\alpha_1\le\hdots\le \alpha_{p-1}\le \alpha_p$ and $\lambda_m\le\alpha_1^2+\hdots+\alpha_p^2\leq \lambda_M$. For $p\leq 3$, the minimal value is $\sqrt{\lambda_m}$, attained at $(\alpha_1,\hdots,\alpha_p)=(0,\hdots,0,\sqrt{\lambda_m})$, while for $p\geq 4$, the minimal value is $2\sqrt{\lambda_m/p}$, attained at $(\alpha_1,\hdots,\alpha_p)=\sqrt{\lambda_m/p}(1,\hdots,1)$.  
Finally we obtain
$$\mathrm{Diam}(SO(N),\Lambda)\leq \pi\lambda_M\lambda_m^{-1/2}\max(2, \lfloor N/2\rfloor^{1/2})$$
A similar method could be used to obtain a bound on the injectivity radius, if we were able to rule out cut points that are not conjugate points on compact semisimple groups.

\subsection{The Zeitlin model}\label{zeitlin}

The incompressible Euler equations of hydrodynamics can be viewed as the geodesic flow on the infinite-dimensional group of volume-preserving diffeomorphisms $\mathrm{Diff}_{\mu}(M)$, where $M$ is the fluid's domain, with respect to the right-invariant $L^2$ metric. For certain two-dimensional manifolds, in particular $M = S^2$ (the two-sphere), the \emph{Zeitlin model}~\cite{Zeitlin} provides a finite-dimensional approximation of the Lie algebra of divergence-free vector fields,  $T_{\Id}\mathrm{Diff}_{\mu}(S^2)$, by a sequence of matrix Lie algebras. This sequence consists of $\mathfrak{su}(N)$ with a scaled Lie bracket and right-invariant Riemannian metric defined as follows. Recall that the Lie algebra $\mathfrak{su}(N)$ admits a decomposition into irreducible $\mathfrak{so}(3)$-modules,
\[
\mathfrak{su}(N) = V_1 \oplus V_2 \oplus \cdots \oplus V_{N-1},
\]
where each $V_\ell$ is $(2\ell + 1)$-dimensional and can be thought of as a finite-dimensional analogue of the $\ell$th spherical harmonic mode. We can define a quantized Laplacian $\Delta_N : \mathfrak{su}(N) \rightarrow \mathfrak{su}(N)$ acting diagonally on this decomposition, assigning eigenvalue $-\lambda_\ell = -\ell(\ell+1)$ to each $V_\ell$ by analogy with the Laplace-Beltrami operator acting on spherical harmonics. The \emph{Zeitlin metric} is then defined as the right-invariant extension of
\begin{equation}\label{zeitlin_metric}
 \langle u, v \rangle := \frac{1}{N}\mathrm{Tr}\big(u^{\dagger}(-\Delta_N)v\big), \qquad u, v \in \mathfrak{su}(N).
\end{equation}
Note that this is no longer bi-invariant, but has the feature that its geometry is a finite-dimensional approximation of $\mathrm{Diff}_{\mu}(S^2)$. For convergence results in the $N \rightarrow \infty$ limit, see \cite{ModinViviani}.

For any matrix $X\in \mathfrak{su}(N)$ with a basis of complex eigenvectors $v_j$ satisfying $X(v_j) = i\alpha_jv_j$, the complex matrices $E_{jk}:= v_jv_k^{\dagger}$ form an eigenbasis of $\ad_X$ with $\ad_X(E_{jk}) = i(\alpha_j-\alpha_k) E_{jk}$. Hence the operator norm of $\ad_X$ is
$$ \|\ad_{X}\|_{op} = \max_{1\le j,k\le N} \lvert \alpha_k-\alpha_j\rvert.$$ 
With $\lambda_m = \tfrac{2}{N}$ and $\lambda_M = (N-1)$, the constraint 
$ \lambda_m \le \sum_{j=1}^N \alpha_j^2 \le \lambda_M$
implies that the infimum is achieved by 
$$ \inf_{\lvert u_0\rvert=1}\|\ad_{\Lambda u_0}\|_{op} = \begin{cases} \sqrt{\frac{8}{N^2}} & N \text{ even,} \\ 
\sqrt{\frac{8}{N^2-1}} & N \text{ odd.}\end{cases}
$$
This infimum is attained by a matrix $X$ which is purely imaginary and diagonal, having half positive and half negative eigenvalues which are all equal to each other in magnitude. 
We end up with diameter bound $$ \mathrm{Diam}(SU(N), \text{Zeitlin}) \le \pi N(N-1)/\sqrt{2}.$$
We can calculate that $\Delta_N$ maps the Cartan subalgebra of such imaginary diagonal matrices to itself, so that every such matrix forms a steady solution of the Euler--Arnold equation, and thus a homogeneous geodesic in $SU(N)$. In principle then the conjugate points can be calculated exactly as in Section \ref{steadygeneral}, but we will not pursue this here.

The Ricci curvature of \eqref{zeitlin_metric} is negative in many directions. In fact, the Ricci tensor is a scalar multiple of the metric on each subspace $V_\ell$, and numerical evidence suggests that roughly one third of these scalar values are negative~\cite{klasleandrosteve}. Despite this, our main theorem shows that \emph{every} geodesic of \eqref{zeitlin_metric} has a conjugate point, since $\mathfrak{su}(N)$ has trivial center, so the conjugate locus entirely encloses the origin. At the same time, it is known that $\mathrm{Diff}_{\mu}(S^2)$ has infinite diameter (\cite{Eliashberg}). This rules out the possibility that the conjugate locus in the Zeitlin model remains uniformly bounded as $N \to \infty$: if all conjugate points lay within a fixed distance $\delta$, independent of $N$, then no geodesic could minimize length beyond $\delta$, and the diameter of $SU(N)$ would be at most $\delta$. This bound would carry over to $\mathrm{Diff}_{\mu}(S^2)$ by the convergence results of \cite{ModinViviani}, contradicting the infinite diameter result. Nevertheless, it is still possible that every geodesic in $\mathrm{Diff}_{\mu}(S^2)$ develops a conjugate point eventually, without a uniform bound on the conjugate time. We leave this as an open question for future work.

The quotient space $SU(N)/SO(3)$ is the configuration space for fluid motion modulo rigid rotations, which is also of interest due to the physical observations that large-scale flows tend towards steady states up to global rotation. Note that geodesics have conjugate points after sufficiently long time in this space by Corollary \ref{homogeneouscorollary}.  On the other hand, Ricci curvature is negative for any fixed vector in the successive approximations for sufficiently large $N$~\cite{klasleandrosteve}.

\section{Techniques for other Lie groups}\label{sec:noncompact}

In this section we analyze three additional techniques that can be used to find conjugate points even when the adjoint action is unbounded.
In Section \ref{closedgeodesicsection} we show that along any nonhomogeneous closed geodesic in a Lie group $G$ with a left-invariant metric, there is a conjugate point. The result also shows that a closed nonhomogeneous geodesic in any homogeneous space $G/H$ must have a conjugate point.  Then in Section \ref{milnornonsteadythm}, we show that for a nonhomogeneous geodesic in a 3D unimodular Lie group, there is a simple expression for the index form found by using a natural orthogonal basis, through the technique of Milnor~\cite{Milnor}. Finally in Section \ref{steadygeneral} we discuss how to solve the constant-coefficient Jacobi equation in the steady case and obtain new examples.

\subsection{Closed nonhomogeneous geodesics}\label{closedgeodesicsection}

The main difficulty in finding conjugate points is choosing an effective direction of variation: one that generates a family of nearby geodesics which, to first order, intersect the original geodesic again at some future time. Our key observation here is that the geodesic itself suggests such a direction: we take the time derivative $u'(t)$ of its left-translated velocity vector, which is never zero for nonhomogeneous geodesics. This provides a canonical solution of the linearized equations, and under a symmetry assumption (that right multiplication by $\gamma(\tau)\gamma(0)^{-1}$ is an isometry), it leads directly to a nontrivial Jacobi field.

Note that this symmetry assumption is automatically satisfied by all closed geodesics. Of course, every closed geodesic has a \emph{cut} point, since it eventually ceases to minimize distance. However, the presence of a conjugate point requires a family of \emph{nearby} geodesics that are shorter.

\begin{theorem}\label{closedgeodesicthm}
Suppose $\gamma(t)$ is a solution of \eqref{eulerarnold}, with velocity $u(t)$ nonconstant. If right multiplication by $\gamma(\tau)\gamma(0)^{-1}$ is an isometry of the left-invariant metric for some $\tau >0$, then $\gamma(\tau)$ is conjugate to $\gamma(0)$. In particular this applies if $\gamma(\tau)=\gamma(0)$.
\end{theorem}

\begin{proof}[Proof of Theorem~\ref{closedgeodesicthm}]
By left-invariance, we may assume without loss of generality that $\gamma(0)$ is the identity. Let $u(t)$ be defined by the flow equation $\gamma'(t) = DL_{\gamma(t)} u(t)$. By~\eqref{eulerarnold_conservation}, $u(t)$ satisfies the angular momentum conservation law
\begin{equation}\label{angularmomentumform}
u(t) = \Adstar_{\gamma(t)}u_0
\end{equation}
in terms of the initial condition $u(0)=u_0$. Differentiating the Euler equation \eqref{eulerarnold} for $u(t)$ with respect to time gives 
$$ u''(t) - \adstar_{u(t)}u'(t) - \adstar_{u'(t)}u(t)=0,$$
showing that $z_p(t) := u'(t)$ is a particular solution of the linearized Euler equation \eqref{jacobiequation}.
Since $u'(t)$ is nowhere zero, this is a nontrivial solution. If we rewrite \eqref{jacobiequation} in the form 
$$ \frac{d}{dt}\Big( \Ad_{\gamma(t)}y(t)\Big) = \Ad_{\gamma(t)} z(t),$$
we see that the general solution of the complementary homogeneous equation is $y_c(t) = \Ad_{\gamma(t)^{-1}}w_0$ for some 
vector $w_0\in \mathfrak{g}$. Thus the general solution is 
$$ y(t) = y_p(t) + y_c(t) = u(t) + \Ad_{\gamma(t)^{-1}}w_0.$$
To have $y(0)=0$ as desired, we choose $w_0 = -u_0$. Inserting \eqref{angularmomentumform} in this, we obtain
\begin{equation*}
y(t) = \Adstar_{\gamma(t)}u_0 - \Ad_{\gamma(t)^{-1}}u_0.
\end{equation*}
If right multiplication by $\gamma(\tau)$ is an isometry, then $\Adstar_{\gamma(\tau)}\Ad_{\gamma(\tau)}=I$,
and we will obtain $y(\tau)=0$. Clearly the corresponding Jacobi field $y(t)$ is nontrivial on $[0,\tau]$ since $z_p(t)$ is nontrivial.
\end{proof}

\begin{remark} If we knew that every   
geodesic $\gamma(t)$ such that 
$\Adstar_{\gamma(t)}\Ad_{\gamma(t)}$ 
met the identity at time $\tau$ was actually periodic in time, this theorem would be a consequence of Theorem~\ref{generaltheorem}. But we do not necessarily know that.\end{remark}

\begin{corollary}\label{homogeneousclosed}
In a homogeneous space $M=G/H$ with left-invariant metric on $G$ satisfying \eqref{condition2}, every closed nonhomogeneous geodesic has a conjugate point. In particular in a nonpositively curved homogeneous space, every closed geodesic must lift to a homogeneous horizontal geodesic in $G$. 
\end{corollary}

\begin{proof}
    Consider a closed geodesic $\gamma\colon [0,\tau]\to M$ and its horizontal lift $\tilde{\gamma}\colon [0,\tau]\to G$. Since $\gamma$ is closed, we have $\gamma(\tau)=\gamma(0)$, and thus $\tilde{\gamma}(\tau)\tilde{\gamma}(0)^{-1}\in H$. Since $H$ acts on the left and right by isometries on $G$ by \eqref{condition2}, it satisfies the conditions of Theorem \ref{closedgeodesicthm}, and there is a conjugate point along $\tilde{\gamma}$ in $G$. Thus by O'Neill's Theorem \ref{oneillhomogeneousindex}, there is also a conjugate point along $\gamma$ in $M$.

    For the second part, if there is a conjugate point then there must be at least some positive curvature in the manifold. Hence if no geodesic has conjugate points, then every closed geodesic must come from a homogeneous horizontal geodesic in $G$, i.e., a steady solution of the Euler--Arnold equation.
\end{proof}

In order to state a corollary of Theorem \ref{closedgeodesicthm}, recall that a Riemannian manifold $M$ is said to have \emph{dense closed geodesics} if for every $p\in M$ and every unit $v\in T_pM$ and every $\varepsilon>0$, there is a $w\in T_pM$ such that $\lvert v-w\rvert < \varepsilon$ and the geodesic $\gamma(t)=\exp_p(tw)$ is closed (i.e., $\gamma(\tau)=p$ for some $\tau>0$). Examples of manifolds with dense closed geodesics include: compact manifolds with negative curvature, by the Anosov theorem; certain quotients of nilpotent Lie groups, all of which have curvatures of both signs (so long as they are not abelian) (\cite{deMeyer}, \cite{Eberlein}, \cite{LeePark}, \cite{Mast}); and of course many examples with positive curvature, such as $U(N)$ with the bi-invariant metric (see also \cite{Besse}). By left-invariance, it is enough to check this condition when $p$ is the identity. 

\begin{corollary}\label{denseclosedgeodesics}
If a Lie group $G$ with left-invariant metric has dense closed geodesics, then either it has positive curvature in some section at the identity, or it is abelian and flat.

If $H=\mathrm{Iso}(G)$ is the isotropy group of $G$ under its left-invariant metric, and if the homogeneous space $M=G/H$ has dense closed geodesics, then it must either have positive curvature in some section at the identity coset, or be naturally reductive. 
\end{corollary}

\begin{proof}
First consider the full Lie group $G$. For every $v\in T_{\id}G$ there is a nearby vector $w$ such that the geodesic in the direction of $w$ is closed. If any such geodesic is nonsteady, the previous theorem gives a conjugate point along it, which implies there must be positive curvature somewhere along the geodesic. Otherwise every closed geodesic is steady. 

If the geodesic $t\mapsto \exp_{\id}(tw)$ is steady, then we must have $\adstar_ww=0$. Now the quadratic form $v\mapsto A(v) := \adstar_vv$ is continuous 
on a finite-dimensional Lie algebra, and density of closed geodesics implies that for every $v\in T_{\id}G$ and $\delta>0$ there is a $w\in T_{\id}G$ such that $\lvert v-w\rvert<\delta$ and $A(w)=0$. Hence we must have $A(v)=0$ for all $v\in T_{\id}G$. 
 
Notice that $\adstar_vv=0$ for all $v\in T_{\id}G$ implies that the metric is bi-invariant, and thus the sectional curvature $K(u,v)$ is given by the well-known formula $K(u,v) = \tfrac{1}{4} \lvert [u,v]\rvert^2$. Thus either $[u,v]=0$ for all $u,v\in T_{\id}G$, so that $G$ is abelian and flat, or we again get some positive curvature.

Now we consider the homogeneous space $M=G/H$, assuming the subgroup $H$ is exactly the isotropy group. In this case every geodesic on $M$ arises from a horizontal geodesic on $G$. If such a geodesic is closed in $M$ and nonhomogeneous, then by Corollary \ref{homogeneousclosed}, it must have a conjugate point, and thus positive curvature in some section along it. Otherwise every closed geodesic in $M$ comes from a homogeneous geodesic in $G$, which by the same reasoning as before implies that $\adstar_ww=0$ for all \emph{horizontal} vectors $w$ in $\mathfrak{g}$. 

This condition implies that for any horizontal vectors $w$ and $x$, we have 
$$ 0 = \langle \adstar_ww, x\rangle = \langle w, \ad_wx\rangle = -\langle \ad_xw, w\rangle. $$
This is exactly the condition for $M$ to be naturally reductive; see Section \ref{homogeneoussection}. Conjugate points can then be analyzed using the technique of Ziller~\cite{Ziller2}.  
\end{proof} 
 
\begin{remark}\label{torus}
This result also gives an alternative proof of the known fact that the only compact Lie group with nonpositive sectional curvature is the abelian flat torus. This arises from the fact that nonpositive curvature requires solvability of the group, and the only compact solvable groups are abelian---see, e.g., \cite{Knapp}. 
\end{remark}

\subsection{The 3D unimodular case}\label{milnornonsteadythm}
 
Consider a three-dimensional unimodular Lie group $G$ with a left-invariant metric. We will establish in Theorem \ref{3dunimodularthm} what the index form \eqref{indexform} looks like in this case along a general nonhomogeneous geodesic, without using the adjoint action. We obtain a criterion that can be computed in terms of the solution of the Euler--Arnold equation, which in this situation would typically involve Jacobi elliptic functions, and for which the geodesic $\gamma(t)$ and corresponding adjoint action $\Ad_{\gamma(t)}$ are impossible to compute explicitly. We apply it in Example \ref{sl2-example} to show that no geodesic in $SL(2,\mathbb{R})$ under the Berger-type metric has conjugate points.

By Milnor~\cite{Milnor}, there is a self-adjoint operator $L\colon \mathfrak{g}\to\mathfrak{g}$ such that $[u,v] = L(u\times v)$ for all $u,v\in \mathfrak{g}$, where $\times$ is the usual three-dimensional cross product on $\mathfrak{g}\cong \mathbb{R}^3$. Note that if $L$ were %is 
invertible, then $\simpkill(u,v) := \langle u, L^{-1}v\rangle$ would be a nondegenerate bi-invariant form; groups that admit such forms are called \emph{quadratic}. However, we do not assume this in what follows. The Euler--Arnold equation is then given by 
\begin{equation}\label{eulerarnoldmilnor}
u' = -u\times Lu
\end{equation}
and we can easily compute that $\langle u,u\rangle$ and $\langle u, Lu\rangle$ are both constant in time. We assume $\langle u,u\rangle = 1$ and define $\ell := \langle u, Lu\rangle$. In the following lemma, we construct a basis which allows us to simplify the index form and isolate the contributions from each component.

\begin{lemma}\label{orthogonallemma}
For any nonsteady solution $u(t)$ of the Euler--Arnold equation, we have that $u'(t)$ is nowhere zero, so the function $\delta(t):=\lvert u'(t)\rvert^2$ is always positive.  The vector fields $u(t)$ along with
$$ v(t) := \delta(t)^{-1} u'(t), \qquad w(t) := Lu(t) - \ell u(t)$$
form an orthogonal (but generally not orthonormal) basis of $\mathfrak{g}$. 
In addition we have the formulas 
\begin{align}
u' + \ad_uu &= \delta  v \label{utadu}\\
v' + \ad_uv &= \xi w \label{vtadv}\\
w' + \ad_uw &= -\ell \delta v,\label{wtadw}
\end{align}
where $\xi := \delta^{-2} \langle u\times Lu, u\times L^2u - L(u\times Lu)\rangle$. 
\end{lemma} 

\begin{proof}
If $u'(t_0)$ were ever zero, then $u(t_0)\times Lu(t_0)$ would be zero, and $u(t_0)$ would be a steady solution of the Euler--Arnold equation, which contradicts uniqueness of solutions since $u$ is assumed nonsteady. Orthogonality of the basis is obvious from the definitions and the conservation laws. Note that $\lvert v\rvert^2 = \delta^{-1}$ and $\lvert w\vert^2 = \delta$ from the definitions and standard cross product formulas.

Equation \eqref{utadu} is obvious from the definitions. Formula \eqref{wtadw} follows from 
\begin{align*} 
w' + \ad_uw &= Lu' - \ell u' + L(u\times (Lu-\ell u)) \\
&= -L(u\times Lu) -\ell u' + L(u\times Lu) = -\ell \delta v.
\end{align*}

Finally to prove \eqref{vtadv}, we first show that the left side is orthogonal to both $u$ and $v$. Since $\langle u,v\rangle = 0$, we have 
$$ 
\langle v'+\ad_uv, u\rangle = -\langle v, u'\rangle + \langle \ad_uv, u\rangle = \langle v, u\times Lu\rangle + \langle L(u\times v), u\rangle = 0,$$
by symmetry of $L$. Furthermore we have 
\begin{align*}
\langle v'+\ad_uv, v\rangle &= \frac{1}{2} \, \frac{d}{dt} \lvert v\rvert^2 + \delta^{-2} \langle L(u\times u'), u'\rangle \\
&= -\delta^{-2} \langle u', u''\rangle + \delta^{-2} \langle L(u\times u'), u'\rangle \\
&= \delta^{-2} \langle u', u\times Lu' + u'\times Lu\rangle + \delta^{-2} \langle u\times u', Lu'\rangle = 0.
\end{align*}
This shows that $v'+\ad_uv = \xi w$ for some function $\xi$. To compute it, we have 
\begin{align*}
\delta^2 \xi &= \delta\langle v'+\ad_uv, w\rangle = -\delta\langle v, w'\rangle + \delta\langle L(u\times v), w\rangle \\
&= -\langle u', Lu' - \ell u'\rangle + \langle u\times u', Lw\rangle \\
&= -\langle u',Lu'\rangle + \ell \delta + \langle w, Lw\rangle \\ 
&= -\langle u',Lu'\rangle + \ell(\lvert Lu\rvert^2 - \ell^2) + \langle Lu, L^2u\rangle - 2\ell \lvert Lu\rvert^2 + \ell^3 \\
&= -\langle u\times Lu, L(u\times Lu)\rangle + 
\langle Lu, L^2u\rangle - \ell \lvert Lu\rvert^2 \\
&= -\langle u\times Lu, L(u\times Lu)\rangle + 
\langle u\times Lu, u\times L^2u\rangle, 
\end{align*}
using $\delta = \lvert u\times Lu\rvert^2 = \lvert Lu\rvert^2 - \ell^2$ and a cross product identity. The desired formula follows.
\end{proof}

With this Lemma, the index form simplifies greatly. 
\begin{theorem}\label{3dunimodularthm}
Suppose $u(t)$ is a solution of the Euler--Arnold equation \eqref{eulerarnoldmilnor}, with $v$ and $w$ defined as in Lemma \ref{orthogonallemma}. 
Let $y(t) = p(t)u(t)+q(t)v(t)+r(t)w(t)$ for functions $p,q,r\colon [0,\tau]\to \mathbb{R}$ which vanish at both $t=0$ and $t=\tau$. Then the index form along the corresponding geodesic is given by 
\begin{equation}\label{indexmilnor}
I(y,y) = \int_0^{\tau} p'^2 + \delta (r' + \xi q)^2 + \delta^{-1} q'^2 - \ell \xi q^2 \, dt.
\end{equation}
In particular, if $\ell \xi(t)\le 0$ for all $t$ then there are no conjugate points along the geodesic, while if both $\delta(t)$ and $\ell \xi(t)$ are bounded below by positive constants $\delta_0$ and $\ell \xi_0$ respectively, then there is a conjugate point no later than $\tau := 2\pi/\sqrt{\ell \xi_0\delta_0}$. 
\end{theorem}

\begin{proof}
Defining $z=y'+\ad_uy$ as usual, we compute using \eqref{utadu}--\eqref{wtadw} that 
\begin{align*}
z &= p'u + q'v + r'w + p(u'+\ad_uu) + q(v'+\ad_uv) + r(w'+\ad_uw) \\
&= p'u + \big( q' + \delta (p- \ell r)\big) v + (r' + \xi q)w.
\end{align*}
On the other hand we have 
$$ Lu\times y = p Lu\times u + q Lu\times v + r Lu \times w = pu' + q(\delta^{-1} \ell w - u) - \ell r u',$$
so that 
$$ z - Lu\times y = (p'+q) u  + q'v + (r'+(\xi-\ell\delta^{-1}) q\big) w.$$
The index form \eqref{indexform} can be expressed in this context as 
\begin{align*} 
I(y,y) &= \int_0^{\tau} \langle z, z-Lu\times y\rangle \, dt \\
&= \int_0^{\tau} p'(p'+q) + \delta^{-1} q'\big(q'+\delta (p-\ell r)\big) + 
\delta (r'+\xi q)(r'+\xi q - \delta^{-1} \ell q) \, dt \\
&= \int_0^{\tau}  p'^2 + p'q + \delta^{-1} q'^2 + q'(p-\ell r)  + \delta (r'+\xi q)^2  - 
\ell q r' - \ell \xi q^2  \, dt.
\end{align*}
Formula \eqref{indexmilnor} follows after the obvious integration by parts, since the functions all vanish at the endpoints.

Obviously $\ell \xi(t)\le 0$ for all time makes the index form positive definite for any nontrivial variation, so there are no conjugate points. 
If we assume $\delta(t)\ge \delta_0>0$ and $\ell \xi(t)\ge \ell \xi_0>0$, then defining $\omega = \sqrt{\ell\delta_0\xi_0}/2$ and setting $p(t)\equiv 0$, 
the index form \eqref{indexmilnor} becomes 
$$ I(y,y) \le \int_0^{\tau} \delta (r' + \xi q)^2 + \delta_0^{-1} (q'^2 - 4\omega^2 q^2) \, dt.$$
Take a variation of the form $q(t) = c_1 \sin{(\omega t)} + c_2 \sin{(2\omega t)}$, which vanishes at $\tau=\pi/\omega$. Choose $r(t)$ so that $r'(t) = -\xi(t) q(t)$; then $r(\tau)$ is some linear combination of $c_1$ and $c_2$, and those coefficients can be chosen to make $r(\tau)$ vanish. We then get a variation for which 
$$I(y,y) \le -\frac{3\pi c_1^2 \omega}{2\delta_0} \le 0,$$
and there must be a conjugate point no later than $\tau$. 
\end{proof}

\begin{example}\label{sl2-example}
Suppose $L$ is diagonalized in an orthonormal basis $\{e_1,e_2,e_3\}$ with
eigenvalues $\lambda_1, \lambda_2, \lambda_3$, with the last two entries the same (corresponding to rotational symmetry in the $e_2$-$e_3$ plane, as in the Berger sphere). Then the Euler--Arnold equation \eqref{eulerarnoldmilnor} for $u = \sum_i u_ie_i$ becomes 
$$ u_1'(t) = 0, \qquad u_2'(t) = -\beta u_1(t)u_3(t), \qquad u_3'(t) = \beta u_1(t)u_2(t), \qquad \beta:=\lambda_1-\lambda_2.$$
By symmetry between $e_2$ and $e_3$ we may assume that $u_3(0)=0$, and the general solution is thus 
$$ u(t) = c_1 e_1 + c_2 q(t), \qquad q(t):=\cos{(\beta c_1 t)}e_2 + \sin{(\beta c_1 t)}e_3.$$
For convenience, define 
$$ \consty:=\beta c_1 c_2, \qquad r(t):= -\sin{(\beta c_1 t)}e_2 + \cos{(\beta c_1 t)} e_3.$$
Note that $u(t)$ is nonsteady if and only if $\consty \ne 0$.
We observe that 
$$ q' = \beta c_1 r, \qquad r' = -\beta c_1 q, \qquad \ad_{e_1}q = \lambda_2 r, \qquad \ad_{e_1}r = -\lambda_2 q, \qquad \ad_qr = \lambda_1 e_1.$$
The orthogonal basis from Lemma \ref{orthogonallemma} then becomes 
$$ u = c_1 e_1 + c_2 q, \qquad 
v = \consty^{-1} r, \qquad w = \consty (c_2e_1 - c_1q), \qquad c_1^2+c_2^2=1, $$
and the quantities needed to apply Theorem \ref{3dunimodularthm} are
\begin{align*}
\delta &= |u'|^2 = |c_2q'|^2 = |c_2 \beta c_1 r|^2 = |c_2 \beta c_1|^2 |r|^2 = D^2, \\ 
\xi &= \frac{1}{\delta^2}\Big(\lambda_1 \beta^2 u_1^2\big(u_2^2 + u_3^2\big)\Big) = \frac{\lambda_1}{D^2}, \quad \text{and} \\
\ell &= \lambda_1c_1^2 + \lambda_2c_2^2
\end{align*}
which are all constant in time.

Hence for example on $SL(2,\mathbb{R})$ where we have $\lambda_2>0$ and $\lambda_1<0$, there are conjugate points along a geodesic with initial velocity $u(0) = c_1e_1 + c_2e_2$ if and only if $\ell = \lambda_1c_1^2 + \lambda_2c_2^2 < 0$.
\end{example}

The same technique can be used more generally on any 
Lie group with a nondegenerate bilinear form in arbitrary dimension, but the construction is still essentially three-dimensional and produces a sufficient but not necessary condition for conjugate points.

\subsection{Homogeneous geodesics on general Lie groups}\label{steadygeneral}

We now turn to homogeneous geodesics: those for which the left-translated velocity vector $\gamma^{-1}(t)\gamma'(t) = u(t)$ is constant. In contrast to the nonhomogeneous case, where we determined sufficient conditions for conjugacy, the first result in this section gives a necessary and sufficient condition, provided that a certain Sylvester-type matrix equation involving the adjoint operator $\adstar_u$ can be solved. This is often the case, especially on  
semisimple groups, as we will see, but we also show in Example \ref{heisenberg-example} that it can work in the non-semisimple case. For such geodesics, this yields an explicit criterion for conjugacy, valid for general left-invariant metrics.

\begin{proposition}\label{steadytheorem}
Suppose $u_0$ is a steady solution of the Euler--Arnold equation \eqref{eulerarnold}, i.e., that $\adstar_{u_0}u_0=0$. Let $\adop$ and $\eulerop$ be the linear operators on the Lie algebra $\mathfrak{g}$ defined by
\begin{equation}\label{jacobi_operators}
\adop(v) := \ad_{u_0}v, \qquad \eulerop(v) = \adstar_{u_0}v + \adstar_vu_0.
\end{equation}
Suppose that there is some subspace $\subby\subset \mathfrak{g}$, such that $\adop$ and $\eulerop$ map $\subby$ to itself and there is an operator $\Syl
\colon \subby\to\subby$ satisfying the Sylvester equation \begin{equation}\label{sylvestercondition}
    \Syl \eulerop+\adop \Syl=I \quad \text{on $\subby$.}
\end{equation}  
Then there is a conjugate point along the geodesic  
if for some $\tau > 0$ we have 
\begin{equation}\label{determinantcondition}
\det{(e^{\tau \adop}\Syl e^{\tau \eulerop}-\Syl)} = 0.
\end{equation}
\end{proposition}

\begin{proof}
Let $u_0\in\mathfrak{g}$ be a steady solution of the Euler--Arnold equation \eqref{eulerarnold}, i.e. such that $\adstar_{u_0}u_0=0$. There are conjugate points along the corresponding geodesic $\gamma$ if there is a time $\tau>0$ and vector fields $y(t)$ and $z(t)$ of the Lie algebra satisfying $y(0) = y(\tau) = 0$ and the system of equations
\begin{equation}\label{jacobisystem_operators}
y'+\adop(y)=z,\qquad z'=\eulerop(z)
\end{equation}
in terms of the linear operators $\adop$ and $\eulerop$ defined in \eqref{jacobi_operators}. 
Now if there is an $\Syl$ satisfying \eqref{sylvestercondition}, then every solution of \eqref{jacobisystem_operators} with $y(0)=0$ and $z(0)=z_0\in \subby$ is obtained as
\begin{equation*}\label{steady_jacobi}
y(t)=\Syl e^{t\eulerop}z_0- e^{-t\adop}\Syl z_0,
\end{equation*} 
since the first part solves the nonhomogeneous equation for $y$, while the second part solves the homogeneous equation. A conjugate point thus occurs at $\tau>0$ for some initial condition $y'(0)\in \subby$ iff $$y(\tau) = e^{-\tau \adop} (e^{\tau \adop}\Syl e^{\tau \eulerop}-\Syl )z_0=0$$
for some $z_0\in \subby$, 
which happens if and only if the determinant of this matrix vanishes as in \eqref{determinantcondition}. 
\end{proof}

\begin{example}\label{heisenberg-example}
Consider the Heisenberg group $\mathbb{H}^3$, whose Lie algebra is spanned by $\{\heis_1, \heis_2, \heis_3\}$, having $[\heis_1, \heis_2] = \heis_3$ as the only nontrivial bracket, and declare $\heis_i$ to be orthonormal. In this case, if $u_0 = \heis_3$, then we can choose $\subby = \mathrm{span}\{\heis_1,\heis_2\}$ so that  
\begin{equation*}
\adop = \mathrm{ad}_{\heis_3} = 0 \qquad \text{and} \qquad \eulerop\vert_{\subby} = \left(
\begin{matrix}
0 & -1 \\
1 & 0
\end{matrix}
\right).
\end{equation*}
Therefore, we can take $\Syl = \eulerop\vert_{\Omega}^{-1}$, and the criterion of Theorem \ref{steadytheorem} becomes
\begin{equation*}
\det\big(\Syl e^{\tau \eulerop\vert_{\Omega}}-\Syl\big) = 2(1-\cos{\tau}), 
\end{equation*}
which means that there are conjugate points along the steady geodesic $t \mapsto \exp(t\heis_3)$,  
with the first conjugate point occurring at $t = 2\pi$. 
\end{example}

Now we consider the special case of semisimple Lie groups, for which the Sylvester equation~\eqref{sylvestercondition} can be solved, as we will see. If $G$ is a semisimple Lie group, then there is a nondegenerate Killing form $\bim(\cdot, \cdot)$ satisfying $\bim(\ad_uv, w) + \bim(v, \ad_uw) = 0$ for all $u,v,w\in \mathfrak{g}$, and any left-invariant  
Riemannian metric $\langle\cdot, \cdot\rangle$ is determined by an invertible symmetric operator $\Lambda$ on $\mathfrak{g}$ such that
\begin{equation}\label{eq_quadratic_metric}
\langle u, v \rangle = \kappa(u, \Lambda v), \qquad \text{for all } u, v \in \mathfrak{g}.
\end{equation}
In this case, the operator $\adstar$ can be written in terms of $\Lambda$ as $\adstar_uv = -\Lambda^{-1}(\ad_u\Lambda v)$, and the Euler--Arnold equation in~\eqref{eulerarnold} reduces to
\begin{equation*}\label{eulerarnoldlambda} 
\Lambda u'(t) + \ad_{u(t)}\Lambda u(t) = 0.
\end{equation*}
Thus $\gamma(t)$ is a homogeneous geodesic if and only if the velocity $u(t)=u_0$ satisfies the steady Euler--Arnold equation $$[u_0,\Lambda u_0] = 0.$$
In this case the operator $\eulerop$ from \eqref{jacobi_operators} can be written 
\begin{equation}\label{Fquadratic}
   \eulerop  
   = \Lambda^{-1}(\ad_{\Lambda u_0} - \ad_{u_0}\Lambda). 
\end{equation}

Generically one expects $\ad_{u_0}$ to be complex-diagonalizable over $\mathfrak{g}$, and whenever $\lambda$ is an eigenvalue, so are $-\lambda$, $\overline{\lambda}$, and $-\overline{\lambda}$. If $u_0$ is a hyperbolic element, then all eigenvalues are real and occur in positive and negative pairs, while if $u_0$ is an elliptic element, then they are all purely imaginary and occur in conjugate pairs. For simplicity we will suppose there is at least one pair of real or imaginary eigenvalues summing to zero, and that the corresponding two-dimensional eigenspace preserved by $\ad_{u_0}$ is also preserved by $\ad_{\Lambda u_0}$. In the same way that the compact adjoint action case can essentially be reduced to  $\mathfrak{so}(3)$ (as in Corollary \ref{easyconjugatethm}), in this case the picture for steady flows with noncompact adjoint action is essentially that of $\mathfrak{sl}(2,\mathbb{R})$.

\begin{theorem}\label{cartan_thm}
Suppose $G$ is a semisimple Lie group with Killing form $\simpkill(\cdot,\cdot)$ and a left-invariant metric $\langle \cdot, \cdot\rangle$ generated by an invertible inertia operator $\Lambda\colon \mathfrak{g}\to\mathfrak{g}$ via \eqref{eq_quadratic_metric}. Let $u_0\in\mathfrak{g}$ with $[u_0,\Lambda u_0]=0$. Suppose there is a two-dimensional subspace $\Omega$ preserved by $\ad_{u_0}$, $\ad_{\Lambda u_0}$, and $\Lambda$, such that $\Tr(\ad_{u_0}\vert_{\Omega}) = 0$ and $\ad_{\Lambda u_0}\vert_{\Omega}$ is not identically zero. 
 
    If $\ad_{u_0}$ has purely imaginary eigenvalues on $\Omega$, then there is a conjugate point along the homogeneous geodesic $\gamma(t) = \exp{(tu_0)}$. If $\ad_{u_0}$ has purely real eigenvalues on $\Omega$, then there is a conjugate point along the geodesic arising from a Jacobi field with $J'(0)\in \Omega$ if and only if $u_0$ is stable in the Eulerian sense under perturbations in $\Omega$. 
\end{theorem}

\begin{proof}
Define $\adop:=\ad_{u_0}\vert_{\Omega}$, $\adlambop:=\ad_{\Lambda u_0}\vert_{\Omega}$, and $C:= \Lambda\vert_{\Omega}$. 
The operator $F$ in \eqref{Fquadratic} then maps $\Omega$ to itself by assumption and takes the form $F = C^{-1}B -C^{-1} A C$. Observe that $C^{-1}\adlambop$ is antisymmetric in the Euclidean inner product since $\adlambop$ is antisymmetric under the bi-invariant inner product. Since it is not identically zero on the two-dimensional subspace, it must be invertible. We can check that $S:= \adlambop^{-1}C$ satisfies the Sylvester equation \eqref{sylvestercondition}, via
$$ SF + \adop S = \adlambop^{-1}CC^{-1}(\adlambop-\adop C) + \adop \adlambop^{-1} C = I, $$
using the fact that $\adop$ and $\adlambop$ commute.

Proposition \ref{steadytheorem} then implies there is a conjugate point along the geodesic if for some $\tau>0$ we have 
 $\det{(e^{\tau\adop}S - S e^{-\tau F})} = 0$.
Since $e^{\tau F} = C^{-1} e^{\tau C F C^{-1}} C = C^{-1} e^{\tau(\adlambop C^{-1}-\adop)} C$, this condition simplifies to 
$$ \det{\big( B^{-1}( e^{\tau \adop} - e^{-\tau (\adlambop C^{-1}-\adop)})C\big)}=0,$$
and since $\adlambop$ and $C$ are invertible on $\Omega$, this is equivalent to 
\begin{equation}\label{detdifference}
\det{(e^{\tau \adop} - e^{-\tau(\adlambop C^{-1}-\adop)})} = 0.
\end{equation}

Since $C^{-1}\adlambop$ is antisymmetric in the Euclidean inner product on $\Omega$, we have that  $\Tr(\adlambop C^{-1}) = \Tr(C^{-1}\adlambop) = 0$, so that both $\adlambop C^{-1}$ and $\adop$ are in $\mathfrak{sl}(2,\mathbb{R})$ as operators on the $2$-dimensional space $\Omega$.
We can then use the following result, proved in the Appendix, to find solutions of \eqref{detdifference}.

\begin{claim}\label{cartan_thm_claim}
    Let $\simpkill(\mattyone,\mattytwo)=\tfrac{1}{2}\Tr(\mattyone\mattytwo)$ denote the Killing form on $\mathfrak{sl}(2,\mathbb{R})$. For any matrices $\mattyone,\mattytwo\in\mathfrak{sl}(2,\mathbb{R})$,
    let $d(t) :=\det{(e^{t\mattyone}-e^{t\mattytwo})}$. If $d(t)$ does not vanish identically, then it has a nonempty set of discrete positive roots if and only if at least one of $\simpkill(\mattyone,\mattyone)$ or $\simpkill(\mattytwo,\mattytwo)$ is negative or, if both are nonnegative, we have
\begin{equation}\label{sl2condition}
2\sqrt{\simpkill(\mattyone,\mattyone)\simpkill(\mattytwo,\mattytwo)} < 2\simpkill(\mattyone,\mattytwo) < \simpkill(\mattyone,\mattyone)+\simpkill(\mattytwo,\mattytwo).
\end{equation}
\end{claim}

If $\adop$ has purely imaginary eigenvalues, then $\simpkill(\mattytwo,\mattytwo) = \simpkill(\adop,\adop) < 0$, and the Claim immediately proves the existence of a conjugate point using Proposition \ref{steadytheorem}. This corresponds to purely Lagrangian stability along $\Omega$. 
    
On the other hand if $\eulerop$ has purely imaginary eigenvalues, corresponding to Eulerian stability in $\Omega$, then  $\simpkill(\mattyone,\mattyone)=\simpkill(\eulerop,\eulerop)<0$. Again the Claim proves existence of a conjugate point. 
      
 Meanwhile, if both $\adop$ and $\eulerop$ have purely real eigenvalues, then 
 we can compute
 $$ \simpkill(\mattyone-\mattytwo,\mattyone-\mattytwo) = \tfrac{1}{2} \Tr\big((\adlambop C^{-1})^2\big) < 0,$$
 so that condition \eqref{sl2condition} cannot be satisfied, and there cannot be conjugate points.
\end{proof}

\begin{example}\label{sl3-example}
As an application of Theorem \ref{cartan_thm}, we give an example of a left-invariant metric on $SL(3)$ for which Ricci is \emph{strictly negative} along the direction of a certain geodesic, but that geodesic still develops conjugate points. The construction of the metric itself is based on \cite{leite}, which we briefly recall here. The Lie algebra $\mathfrak{sl}(3)$ has a natural decomposition as
\begin{equation}\label{eq_sl_brackets}
\mathfrak{sl}(3) = \mathfrak{so}(3) \oplus \mathfrak{sm}(3) \oplus \mathfrak{t}
\end{equation}
with the Killing form  
being negative-definite on $\mathfrak{so}(3)$ and positive-definite on the other two subspaces; here $\mathfrak{t}$ is the Cartan subalgebra, corresponding to traceless \emph{diagonal} matrices, and $\mathfrak{sm}(3)$ consists of symmetric matrices which vanish on the diagonal. The bracket structure is the following:
\begin{alignat*}{5}
[\mathfrak{t}, \mathfrak{t}] &= 0, \qquad
&[\mathfrak{t}, \mathfrak{sm}(3)] &= \mathfrak{so}(3), &\qquad [\mathfrak{t}, \mathfrak{so}(3)] &= \mathfrak{sm}(3) \\ 
[\mathfrak{sm}(3), \mathfrak{sm}(3)] &= \mathfrak{so}(3),
\qquad
&[\mathfrak{sm}(3), \mathfrak{so}(3)] &= \mathfrak{sm}(3) \oplus \mathfrak{t}.&&
\end{alignat*}
Define $\Lambda : \mathfrak{sl}(3) \rightarrow \mathfrak{sl}(3)$ by
\begin{equation*} 
\Lambda|_{\mathfrak{so}(3)} = \rho \cdot \mathrm{Id}_{\mathfrak{so}(3)}, \quad
\Lambda|_{\mathfrak{sm}(3)} = \beta \cdot \mathrm{Id}_{\mathfrak{sm}(3)}, \quad
\Lambda|_{\mathfrak{t}} = \alpha \cdot \mathrm{Id}_{\mathfrak{t}}.
\end{equation*}
That is, $\Lambda$ acts as a scalar multiple of the identity in each subspace of \eqref{eq_sl_brackets}. If we choose, for example, $\rho = \frac{4}{21}$, $\beta = 1$ and $\alpha = \frac{8}{15}$, then the lemma preceding the main theorem of \cite{leite} shows that $SL(3)$ equipped with the left-invariant metric given at the Lie algebra by
\begin{equation*}
\langle u, v \rangle = \tfrac{1}{2} \mathrm{Tr}\big((\Lambda u)\cdot v^{\top}\big), \qquad u, v \in \mathfrak{sl}(3)
\end{equation*}
has strictly negative Ricci curvature for vectors in the subspace $\mathfrak{t}$. Now let
\begin{equation*}
u_0 =
\left(\begin{matrix} 
1 & 0 & 0 \\
0 & 0 & 0 \\
0 & 0 & -1
\end{matrix}\right), \quad 
v_1 =
\left(\begin{matrix} 
0 & 0 & 0 \\
0 & 0 & -1 \\
0 & 1 & 0
\end{matrix}\right),
\quad
v_2 =
\left(\begin{matrix} 
0 & 0 & 0 \\
0 & 0 & 1 \\
0 & 1 & 0
\end{matrix}\right).
\end{equation*}
Note that $u_0 \in \mathfrak{t}$ is a regular element of the Lie algebra, and that $\ad_{u_0}$ preserves the subspace $\Omega := \mathrm{span}\{v_1, v_2\}$. We have
\begin{equation*}
\adop = \ad_{u_0}|_{\Omega} = \left(\begin{matrix} 
0 & -1 \\
-1 & 0
\end{matrix}\right), \qquad 
\mattyone = \adlambop C^{-1}-\adop = \left(\begin{matrix} 
0 & 1-\alpha \\
1-\frac{\alpha}{\rho} & 0
\end{matrix}\right),
\end{equation*}
so that
\begin{equation*}
\bim(\mattyone,\mattyone) = 
\tfrac{1}{2} \mathrm{Tr}\big(\mattyone^2\big) = \frac{
(\alpha-1)(\alpha-\rho)}{\rho},
\end{equation*}
which is negative by our choice of $\alpha$ and $\rho$. Thus, the geodesic in the direction of $u_0$ has conjugate points by Theorem \ref{cartan_thm}. This shows in particular that the Gromoll--Meyer condition is not necessary for conjugate points.
\end{example}

\section{Conclusion \& Outlook}

For Lie groups where the adjoint action is uniformly bounded below in operator norm, we have obtained a complete characterization of the initial velocities $u_0 \in \mathfrak{g}$ for which the corresponding geodesic $\gamma$ develops a conjugate point: this occurs if and only if $u_0$ is not orthogonal to the commutator subalgebra $[\mathfrak{g}, \mathfrak{g}]$. At the same time, as shown by Example \ref{ex_no_conjugate}, this criterion can fail without the assumption of uniform boundedness of $\mathrm{Ad}_{\gamma(t)}$. A natural next step is to identify precisely when this uniform boundedness holds, as this might yield a full description of the conjugate locus on Lie groups in terms of purely algebraic data.

As illustrated in Section \ref{rigid-body} for the case of rigid bodies, the explicit bounds on conjugate times in Corollary \ref{easyconjugatethm} can be used to obtain bounds on the diameter for any left-invariant metric on $G$, provided one can effectively calculate certain operator norms. It would be of interest to obtain bounds for the injectivity radius as well, and to carry out similar analysis on other groups, particularly $SU(N)$ due to its connection to fluid dynamics via quantization.

Finally, one can likely weaken the assumptions of  Theorem \ref{cartan_thm} regarding homogeneous geodesics with steady Euler velocity $u_0$. For example, if $\ad_{u_0}$ has a complex eigenvalue $\alpha + i\beta$ with $\alpha, \beta \neq 0$, one could work in the four-dimensional subspace corresponding to the four eigenvalues $\pm \alpha \pm i\beta$ and obtain a similar, four-dimensional version of Claim \ref{cartan_thm_claim}. In addition, if we use the index form \eqref{indexform} rather than the explicit solution of the Jacobi equation, it would likely not be necessary to assume $\Lambda$ maps an eigenspace to itself. This would enable a more complete description of the connection between stability in the Eulerian and Lagrangian senses and existence of conjugate points. It would also be interesting to relate stability to conjugate points for nonsteady Euler velocities $u(t)$ describing nonhomogeneous geodesics.

\appendix
\section{Proof of Claim~\ref{cartan_thm_claim}}

Here we prove the claim used in the proof of Theorem~\ref{cartan_thm}, which involves a case by case analysis of the possible values the Killing form can take on $\mattyone,\mattytwo\in \mathfrak{sl}(2,\mathbb{R})$.

\begin{claim} 
    Let $\simpkill(\mattyone,\mattytwo)=\tfrac{1}{2}\Tr(\mattyone\mattytwo)$ denote the Killing form on $\mathfrak{sl}(2,\mathbb{R})$. For any matrices $\mattyone,\mattytwo\in\mathfrak{sl}(2,\mathbb{R})$,
    let $d(t) :=\det{(e^{t\mattyone}-e^{t\mattytwo})}$. If $d(t)$ does not vanish identically, then it has a nonempty set of discrete positive roots if and only if at least one of $\simpkill(\mattyone,\mattyone)$ or $\simpkill(\mattytwo,\mattytwo)$ is negative or, if both are nonnegative, we have
\begin{equation}\label{sl2conditionsecondtime}
2\sqrt{\simpkill(\mattyone,\mattyone)\simpkill(\mattytwo,\mattytwo)} < 2\simpkill(\mattyone,\mattytwo) < \simpkill(\mattyone,\mattyone)+\simpkill(\mattytwo,\mattytwo).
    \end{equation}
\end{claim}

\begin{proof}
Note that $d(t)$ is obviously an analytic function of time, so it either vanishes identically or has a set of isolated positive roots. Hence we only need to determine whether this set is nonempty. Note that $d(0)=d'(0)=0$ in all cases.

\textbf{Case 1: $\simpkill(\mattyone,\mattyone)=a^2$, $\simpkill(\mattytwo,\mattytwo)=b^2$, $a,b>0$.} 

The matrix exponentials are given by 
    $$ e^{t\mattyone} = (\cosh{at}) I + a^{-1}(\sinh{at}) \mattyone, \qquad e^{t\mattytwo} = (\cosh{bt}) I + b^{-1}(\sinh{bt}) \mattytwo. $$
    It is then straightforward to compute that 
\begin{equation*}
    d(t) = 2\big( 1-\cosh{at}\cosh{bt} + D \sinh{at}\sinh{bt}\big), \qquad D:=\simpkill(\mattyone,\mattytwo)/(ab), 
    \end{equation*}  

If $D\le 1$, then for all $t>0$, we have 
    $$ d(t) \le 2(1-\cosh{at}\cosh{bt} + \sinh{at}\sinh{bt}) = 2\big(1 - 
    \cosh{\big((a-b)t\big)}\big) \le 0. $$
    Here either $D=1$ and $a=b$ so that $d(t)$ vanishes identically, or $d(t)$ is strictly negative for $t>0$.
    If $D\ge \tfrac{1}{2ab}(a^2+b^2)$ then we can compute that \begin{align*}
        d''(t) &= 2(2Dab-a^2-b^2) \cosh{at}\cosh{bt} + 2\big( (a^2+b^2)D - 2ab\big) \sinh{at}\sinh{bt} \\
        &\ge \frac{(a^2-b^2)^2}{ab} \,\sinh{at}\sinh{bt}\ge 0,
        \end{align*}
    so $d(t) \ge 0$ for all $t>0$, and is strictly positive unless $a=b$ and $D=1$, when it vanishes identically.
    
    On the other hand if $1<D<\tfrac{1}{2ab}(a^2+b^2)$, then 
    $$d''(0) = 2(2abD-a^2-b^2) < 0,$$
    so that $d(t)$ is negative for small positive $t$,
    while for large $t$ we have $d(t) \approx \tfrac{1}{2}(D-1)e^{(a+b)t}>0$. Hence there must be a change of sign for some positive $t$.

    \textbf{Case 2: $\simpkill(\mattyone,\mattyone)=-a^2$, $\simpkill(\mattytwo,\mattytwo)=b^2$, $a,b>0$.} 
    
    Then we have 
    $$d(t) = 2(1-\cos{at} \cosh{bt} + D\sin{at} \sinh{bt}), \qquad D := \simpkill(\mattyone,\mattytwo)/(ab). $$
    For $\tau:=\pi/a$, we clearly have $d(\tau)>0$ and $d(2\tau)<0$, so $d$ changes sign between $\tau$ and $2\tau$.

    \textbf{Case 3: $\simpkill(\mattyone,\mattyone)=-a^2$, $\simpkill(\mattytwo,\mattytwo)=-b^2$, $a,b>0$.}  

    We can write
    \begin{align*}
        d(t) &= 2-2\cos{a t}\cos{bt} +2 D \sin{at}\sin{bt}, \qquad D:=\simpkill(\mattyone,\mattytwo)/(ab) \\
        &= 
        2(1+D) \sin^2{\left( \frac{a+ b}{2} \, t\right)} + 2(1-D)\sin^2{\left( \frac{ b -a}{2} \, t\right)}.
    \end{align*}
    Note that $\lvert D\rvert\ge 1$ in this case with equality if and only if $\mattytwo$ is a multiple of $\mattyone$. To see this, apply conjugation to get a canonical form for $\mattyone$, using invariance of the Killing form, to assume without loss of generality that
    $$ \mattyone = \left( \begin{matrix} 0 & r \\ 
    -r & 0 \end{matrix}\right), \qquad \mattytwo = \left(\begin{matrix} x & y+z \\ y-z & -x\end{matrix}\right)$$
    for some reals $r,x,y,z$. 
    Then we have that 
    $$ r^2=a^2, \qquad \simpkill(\mattyone,\mattytwo) = -rz, \qquad 
    x^2+y^2-z^2= -b^2.$$
    Hence $$\lvert D\rvert = \lvert rz/ab\rvert = \lvert z/b\rvert \ge 1,$$
    with equality iff $x=y=0$, in which case $\mattytwo$ is a multiple of $\mattyone$.

If $a=b$ then the second term in $d(t)$ vanishes, and the function either vanishes identically or only for integer multiples of $\pi/a$.
Otherwise defining $\tau_1 := \frac{2\pi}{b+a}$ and $\tau_2 := \frac{2\pi}{b-a}$, we will have 
$$d(\tau_1)=2(1-D)\sin^2\left(\frac{b-a}{b+a}\pi\right) \ge 0 \quad\text{and}\quad d(\tau_2)=2(1+D)\sin^2\left(\frac{a+b}{b-a}\pi\right) \le 0,$$    
with at least one nonzero.
Thus $d(t)$ vanishes between $\tau_1$ and $\tau_2$.

\textbf{Case 4: $\simpkill(\mattyone,\mattyone)=0$, $\simpkill(\mattytwo,\mattytwo)=b^2$, $b>0$.}

In this case we have 
$$ d(t) = 2(1-\cosh{bt} + Dt \sinh{bt}), \qquad D:= \bim(\mattyone,\mattytwo)/b.$$
The same reasoning as in Case 1 shows that this can vanish if and only if $0<2D<b$, which is precisely condition \eqref{sl2conditionsecondtime}.

\textbf{Case 5: $\simpkill(\mattyone,\mattyone)=0$, $\simpkill(\mattytwo,\mattytwo)=-b^2$, $b>0$.}

In this case we have 
$$d(t) = 2(1-\cos{bt} + Dt\sin{bt}) = 4\sin{(\tfrac{bt}{2})} \big( Dt \cos{(\tfrac{bt}{2})} + \sin{(\tfrac{bt}{2})}\big), \qquad D:= \bim(\mattyone,\mattytwo)/b.$$
This obviously vanishes at all integer multiples of $2\pi/b$. 

\textbf{Case 6: $\simpkill(\mattyone,\mattyone)=0$, $\simpkill(\mattytwo,\mattytwo)=0$.}

In this last case we have 
$$d(t) = 2t^2 \bim(\mattyone,\mattytwo),$$
which obviously is either always zero or never zero for $t>0$;
also the strict inequality of condition \eqref{sl2conditionsecondtime} cannot be satisfied.
\end{proof}

\end{document}